\documentclass[11pt, a4paper]{article}

\usepackage{graphicx,latexsym, amsmath, amsfonts, amssymb, amsthm, url}

\usepackage{tikz}
\definecolor{dgreen}{rgb}{0,0.6,0}
\definecolor{dbrown}{rgb}{0.45,0.25,0}

\newcommand{\gap}{\vspace{0.1in}}

\newcommand{\wt}{\widetilde}

\newcommand{\ol}{\overline}

\newcommand{\hatb}{\hat b}

\newtheorem{theorem}{Theorem} [section] 
\newtheorem{proposition}{Proposition} [section] 
\newtheorem{remark}{Remark} [section]

\begin{document}


\title{Minimax Optimal  Estimation of Convex Functions in the Supreme Norm}
%
%

\author{Teresa  M.   Lebair,\footnote{Department of Mathematics and Statistics, University of Maryland Baltimore County, Baltimore, MD 21250, U.S.A. Email:
ei44375@umbc.edu.}  \  \ \ Jinglai  \, Shen,\footnote{Department of Mathematics and Statistics, University of Maryland Baltimore County, Baltimore, MD 21250, U.S.A. Email:
shenj@umbc.edu.} \  \  \ and \  \  \ Xiao \, Wang\footnote{Department of Statistics, Purdue University, West Lafayette, IN 47907, U.S.A. E-mail: wangxiao@purdue.edu.}}


\maketitle

\begin{abstract}
%
Estimation of convex functions finds broad applications in engineering and science, while  
 convex shape constraint gives rise to numerous challenges in asymptotic performance analysis.
This paper is devoted to minimax optimal estimation of univariate convex functions from  the H\"older class  in the framework of shape constrained nonparametric estimation. 
Particularly, the paper establishes the optimal rate of convergence in two steps   for the minimax sup-norm risk of convex functions  with the H\"older order between one and two. In the first step, by applying information theoretical results on probability measure distance, 
%
%
we establish the minimax lower bound under the supreme norm by constructing a novel family of piecewise quadratic convex functions in the H\"older class.  
In the second step, we develop a penalized convex spline  estimator and establish the minimax upper bound  under the supreme norm.  Due to the convex shape constraint, the optimality conditions of penalized convex splines are characterized by nonsmooth complementarity conditions. 
By exploiting complementarity methods, a critical uniform Lipschitz property of optimal spline coefficients in the infinity norm is established.  This property, along with asymptotic estimation  techniques,  leads to  uniform bounds for bias and stochastic errors on the entire interval of interest. This further yields  the optimal rate of convergence by choosing  the suitable number of knots and penalty value.  The present paper provides the first rigorous justification of the optimal minimax risk for convex estimation under the supreme norm.

{\it Key Words}:  shape constrained estimation,  convex regression, minimax estimation theory,  sup-norm risk, penalized splines, asymptotic analysis, complementarity conditions. 
\end{abstract}

%
\section{Introduction}

Nonparametric estimation of shape constrained functions (e.g., monotone/convex functions) receives increasing attention \cite{EgMartin_book10, PW_Sinica07,  SWang_ACC10, SWang_CDC11,  WangShen_Biometrika09}, driven by a wide range of applications in science and engineering. Examples include reliability engineering, biomedical research, finance, and astronomy.  The goal of shape constrained estimation is to develop an estimator that preserves a pre-specified shape property of a true function, e.g., the monotone or convex property. A challenge in shape constrained estimation is that an estimator is subject to {\em inequality}  shape constraints, e.g., the monotone or convex constraint. These constraints lead to nonsmooth 
conditions in estimator characterizations and complicate asymptotic performance analysis.

Considerable progress has been made  toward developing and analyzing shape constrained estimators in the framework of nonparametric estimation.  For example, estimators that preserve the monotone property have been extensively studied in the literature, e.g., \cite{mammen_99,    
 PW_Sinica07, SWang_ACC10, SWang_SICON11,  utreras_85, WangShen_Biometrika09, wright_81}.
In the realm of  convex (or concave) estimation, earlier research is focused on the least squares approach: the least squares convex estimator is first studied in \cite{hildreth_54} and is shown  to be consistent in the interior of the interval of interest \cite{hanson_76}.  The pointwise rate of convergence for the least squares convex estimator is developed in \cite{mammen_91} and pointwise asymptotic distributions are characterized in \cite{groeneboom_01}. A pointwise convex estimation approach is recently introduced in \cite{CaiLow_12}; this approach is non-asymptotic and focuses on the performance of each individual function.  To deal with unknown information of a function class, adaptive convex estimation has been proposed in \cite{dumbgen_03, dumbgen_04}, which, however, can only handle a compact sub-interval without including the boundary.  To overcome this problem,  adaptive convex regression splines are proposed in \cite{WangShen_SICONConvex12}. 
 Other results include \cite{meyer_08, SWang_CDC11, SWang_ACC12}.

%
%

Given a function class $\Sigma$, several critical questions arise when evaluating asymptotic performance of estimators over the function class $\Sigma$: 
\begin{enumerate}
   \item [(1)] What is the ``best'' rate of convergence of estimators uniformly on $\Sigma$?
   \item [(2)] How can one construct an estimator  that achieves the ``best'' rate of convergence on $\Sigma$?
   \item [(3)] Is the ``best'' rate of convergence in (1) strict on $\Sigma$ for any permissible estimator?
\end{enumerate}
These questions form central research issues in minimax theory of  nonparametric estimation \cite{Nemirovski_notes00, Tsybakov_book10}. In particular, the first and second questions pertain to the minimax upper bound on $\Sigma$ and its estimator construction, and the third question is closely related to the minimax lower bound on $\Sigma$. For {\em unconstrained} estimation, the above questions have been satisfactorily addressed for both the Sobolev and H\"older classes under the  $L_2$ or supreme norm; see \cite{Nemirovski_notes00, Tsybakov_book10} and references therein for details. This has led to well known optimal rates of convergence  over  unconstrained function classes (cf. (\ref{eqn:converge_rate}) of Section~\ref{sect:main}).  However, if shape constraints are imposed, then minimax asymptotic analysis becomes  highly complicated and much fewer results have been reported; some exceptions include \cite{low_02} for monotone estimation and \cite{CaiLow_12, dumbgen_03} for convex estimation. It is known that a shape constraint does not improve the  unconstrained  optimal rate of convergence \cite{kiefer_82}, and it is believed that the same optimal rate holds on a constrained function class but no rigorous justification has been given for general constraints. %
Furthermore, it is worth pointing out that when the supreme norm (or simply sup-norm) is considered, shape constrained minimax analysis becomes even more challenging due to the following reasons:
\begin{enumerate}
   \item [(a)] Constructive minimax upper bound. Most shape constrained estimators  in the literature lack uniform convergence on the entire interval of interest, and they demonstrate poor performance on the boundary of the interval.    For example, the shape constrained least squares estimator is inconsistent at the boundary points \cite{woodroofe_93}, and the convex estimators developed in \cite{dumbgen_03, dumbgen_04}  deal with a compact sub-interval without including the boundary. Other pointwise estimators, such as \cite{CaiLow_12, low_02},   are  applicable only to a small interval of a fixed interior point, and they do not yield a convex or monotone estimate on the entire interval.  This hinders construction of an estimator that achieves the optimal minimax risk under the sup-norm. On the other hand,  in order to establish the uniform convergence requested by the sup-norm risk, one inevitably  faces  many nontrivial issues arising from an underlying nonsmooth optimization formulation of a shape constrained estimator, which call for new tools to handle them.
   
   \item [(b)] Minimax lower bound.  By minimax lower bound theory, which is based upon  information theoretical results  on distance between probability measures, it is known  that establishing a minimax lower bound amounts to constructing a family of  functions (or hypotheses) from a function class satisfying a suitable sup-norm separation order  and a small total $L_2$-distance order \cite[Section 2]{Tsybakov_book10}.  While it is conceived that there exist many such families, a shape constraint considerably  limits choice of  a feasible family under the sup-norm, especially when a higher order shape constraint is imposed (recalling that roughly speaking, the convex constraint places a second order constraint on a function). Therefore, great care needs to be taken in order to meet both the order conditions and shape constraints. 
%
\end{enumerate}

The present paper is devoted to minimax optimal estimation of univariate convex functions from  the H\"older class with H\"older order $r\in (1, 2]$ under the sup-norm. Specifically,  we develop a two-step procedure to establish the optimal rate of convergence of the sup-norm risk.  In the first step, we construct a family of convex functions from the H\"older class that yields the minimax lower bound in the sup-norm. In the second step, a penalized spline based convex estimator is developed and is shown to achieve the minimax upper bound in the sup-norm.  It should be mentioned that even though the obtained optimal rate of convergence coincides  with the optimal rate for the {\em unconstrained} H\"older class as expected, its proof is much more involved than that of the unconstrained case.
%
%
%
In order to overcome  shape constraint induced difficulties,  several new techniques from asymptotic estimation and complementarity theory in constrained optimization are invoked. These new techniques and major contributions of the paper  are summarized below. 

1. Minimax lower bound. Toward this end, we construct a family of piecewise quadratic convex functions (whose derivatives are increasing and piecewise linear); see Section~\ref{sect:lower_bd}. These functions overlap on most of the interval $[0, 1]$, except on certain small sub-intervals. By careful selection of the slopes of the derivatives of these functions and the length of non-overlapping sub-intervals, we show that the constructed convex  functions satisfy the desired order conditions, thus leading to the minimax lower bound. To the best of our knowledge, this construction is the first of its kind for minimax convex estimation.  The proposed construction process also sheds light on minimax lower bounds for monotone or higher order derivative constraints.

2. Constructive minimax upper bound.  We consider a convex penalized spline (or simply $P$-spline) estimator  subject to the second order difference penalty. The convex shape constraint is converted to the second order difference constraint on spline coefficients.  In spite of numerical advantages as well as conceptual simplicity and flexibility of $P$-splines  \cite{LiRuppert_08, marx_96}, optimal spline coefficients of the convex $P$-spline estimator are characterized by nonsmooth complementarity conditions, thanks to the convex shape constraint.  The present paper distinguishes the study of  $P$-spline estimators from that in the literature by establishing a critical uniform Lipschitz property of optimal spline coefficients in the infinity norm \cite{SWang_ACC10, SWang_SICON11}, inspired by uniform convergence required by the sup-norm risk analysis. The proof of uniform Lipschitz property makes extensive use of combinatorial arguments in complementarity theory and the properties of an underlying constrained optimization problem, e.g., the second order difference penalty and piecewise linear formulation of optimal spline coefficients;   see Section~\ref{sec:uniform_Lips}.  
 By exploiting the uniform Lispschitz property and asymptotic estimation techniques,  we  develop uniform bounds for bias and stochastic errors on the entire interval $[0, 1]$ over the H\"older class.  These results pave the way for the desired minimax upper bound under the sup-norm.

The paper is organized as follows. In Section~\ref{sect:main}, we present main results of the paper.  Section~\ref{sect:lower_bd} establishes the minimax lower bound for the rate of convergence under the sup-norm. In Section \ref{sect:upper_bound}, we develop a convex $P$-spline estimator and show that this estimator attains the optimal rate of convergence. The paper ends with concluding remarks in Section \ref{sect:conclusion}.

%
\section{Problem Formulation and Main Results}   \label{sect:main}

Consider the convex regression problem:
\begin{equation}  \label{eqn:convex_estimation}
y_i = f(x_i) + \sigma  \epsilon_i, \qquad i=1, \ldots, n,
\end{equation}
where $f:[0, 1]\rightarrow \mathbb R$ is an unknown convex function, 
 the constant $\sigma>0$, $\epsilon_i$ are independent, standard normal errors, $x_i=i/n, \, i=1, \ldots, n$ are the equally  spaced design points.  Let $f'(\cdot)$ denote the derivative of $f$. Let
\[
{\cal C} :=\Big\{f: [0, 1]\rightarrow \mathbb R \, \Big | \, \big( f'(x_1) - f'(x_2) \big) \cdot \big(x_1 - x_2 \big) \ge 0, \quad \forall \, x_1, x_2 \in [0, 1] \Big\}
%
\]
 be the collection of continuous convex functions which are differentiable (almost everywhere) on $[0, 1]$, and
 $H^r_L$ be the H\"older class with the H\"older exponent (or order) $r\in (1, 2]$ and the H\"older constant $L>0$, namely,
\[
H^r_L := \Big\{f:[0,1] \rightarrow \mathbb R \, \Big | \, |f{'}(x_1)-f{'}(x_2)|\le L |x_1-x_2|^{\gamma}, ~~\forall ~ x_1, x_2\in [0,1]\Big\},
\]
where $\gamma := r-1\in (0, 1]$. Furthermore, let ${\cal C}_H(r, L) : = {\cal C} \cap H^r_L$ be the collection of functions in both ${\cal C}$ and $H^r_L$.

For estimation of {\em unconstrained} functions over the  H\"older class $H_L^r$,
it is well known in minimax theory of nonparametric estimation that
for a fixed $r$, there exists an estimator (depending on $r$) which achieves the optimal rate of convergence over $H_{L}^r$  in the sup-norm \cite{stone_82, Tsybakov_book10}. In fact, the minimax sup-norm risk on $H_L^r$ has an asymptotic order given by
 \begin{equation} \label{eqn:converge_rate}
\inf_{\hat f}\sup_{f\in H_L^r} \mathbb E\big (\|\hat f - f\|_\infty\big) \, \asymp \, L^{1\over 2r+1} \sigma^{2r\over 2r+1} \Big({\log n\over n}\Big)^{r\over 2r+1},
\end{equation}
where $\hat f$ denotes an estimator of a true function $f$, $\mathbb E(\cdot)$ is the expectation operator, and $a\asymp b$ means that $a/b$ is bounded by two positive constants from below and above for all $n$ sufficiently large. The goal of this paper is to establish the same asymptotic minimax rate on ${\cal C}_H(r, L)$ with $r\in (1, 2]$.
Specifically, the main result of this paper is presented in the following theorem.

\begin{theorem}\label{thm:main_result}
  Let $r\in (1, 2]$ and ${\cal C}_{H}(r, L)$ denote the H\"older class of convex functions defined above. Then there exists a positive constant $C_0$  such that
  \begin{equation} \label{eqn:optimal_rate}
    \inf_{\hat f}\sup_{f\in \mathcal C_H(r, L) } \mathbb E\big (\|\hat f - f\|_\infty\big) \, \asymp \,  C_0 \Big({\log n\over n}\Big)^{r\over 2r+1}.
  \end{equation}
  \end{theorem}

The proof of Theorem~\ref{thm:main_result} is divided into two sections. Section~\ref{sect:lower_bd} establishes a lower bound of the optimal rate of convergence  via a construction procedure based on multiple hypotheses originating from information theory \cite[Theorem 2.5]{Tsybakov_book10}. Section~\ref{sect:upper_bound} develops a penalized B-spline based convex estimator that achieves  the optimal rate in (\ref{eqn:optimal_rate}); this gives rise to an upper bound of the optimal convergence rate and, along with the lower bound, yields the desired optimal rate in (\ref{eqn:optimal_rate}).

%
\section{Minimax Lower Bound of Convex Estimators} \label{sect:lower_bd}

In this section, we establish the minimax lower bound of convex estimation of functions in ${\cal C}_H(r, L)$ with $r\in (1, 2]$ in the sup-norm. The key idea of developing such a lower bound for nonparametric estimators relies on tools for distance of multiple probability measures or  hypotheses from information theory \cite{Birge_TIT05, CoverT_book05, Gallager_book68, Pinsker_book64}; see \cite[Section 2]{Tsybakov_book10} for detailed discussions.
It follows from minimax theory (e.g., \cite[Theorem 2.5]{Tsybakov_book10}) that 
establishing a minimax lower bound over the function class ${\cal C}_H(r, L)$ in the sup-norm boils down to the construction of a family of functions (or hypotheses) $f_{j, n}, j=0, 1, \ldots, M_n$ satisfying the following three conditions:
\begin{itemize}
   \item[(C1)] each $f_{j, n} \in {\cal C}_H(r, L)$, $j=0, 1, \ldots, M_n$;
   \item [(C2)] whenever $j \ne k$, $\| f_{j, n} - f_{k, n} \|_\infty \ge 2 s_n>0$, where $s_n  \asymp (\log n/n)^{r/(2r+1)}$;
   \item [(C3)]  there exists a fixed constant $c_0 \in (0, 1/8)$ such that for all $n$ sufficiently large,
   \[
      \frac{1}{M_n} \sum_{j=1}^{M_n} K(P_j, P_0)\, \leq \, c_0 \log(M_n),
   \]
   where $P_j$ denotes the distribution of $(Y_{j, 1}, \ldots, Y_{j, n})$, where $Y_{j, i}=f_{j, n}(X_i) + \xi_i, \, i=1, \ldots, n$  with $X_i=i/n$ and $\xi_i$ being iid random variables, and $K(P, Q)$ denotes the Kullback divergence between the two probability measures $P$ and $Q$ \cite{Kullback_TIT67}, i.e.,
\[
   K(P,Q) :=  \left\{
   \begin{array}{lr} \displaystyle \int \log \frac{dP}{dQ} \, dP, & \text{ if } P \ll Q \\
       +\infty, & \text{otherwise}
   \end{array}
 \right. .
\]
\end{itemize}
%
%
In addition, we assume that there exists a constant $p_*>0$ (independent of $n$ and $f_{j, n}$) such that $K(P_j, P_0) \le p_* \sum^n_{i=1} \big( f_{j, n}(X_i) - f_{0, n}(X_i) \big)^2$. This assumption holds true if the iid random variables $\xi_i \sim N(0, \sigma^2)$ (cf.  \cite[(2.36)]{Tsybakov_book10} or \cite[Section 2.5, Assumption B]{Tsybakov_book10}). Hence, the convex regression problem defined in (\ref{eqn:convex_estimation}) satisfies this assumption.

In other words, once a family of functions  $\{ f_{j, n} \}$ satisfying the above  three conditions is constructed, then the following minimax lower bound over ${\cal C}_H(r, L)$
will hold:
\begin{equation} \label{eqn:lower_risk}
    \liminf_{n \rightarrow \infty} \inf_{\hat f_n} \sup_{f\in \mathcal C_H(r, L)} \Big( {n \over \log n}\Big)^{{r \over 2r + 1}} \, \mathbb E (\| \hat f_n - f \|_\infty) \ge c
\end{equation}
for some constant $c>0$ depending on $r, L$, and $p_*$ only, where  $\inf_{\hat f_n}$ denotes the infimum over all convex estimators on $[0, 1]$. In view of this, the goal of this section is to construct a family of suitable functions $f_{j, n}$ satisfying (C1)-(C3).

%
\subsection{Construction of the Desired Functions $f_{j, n}$ }
\label{subsect:construction_func}

Consider the function class $\mathcal C_H(r, L)$ with $r\in (1, 2]$ and $L>0$, and  fix $c_0 \in (0,\frac{1}{8})$.  Given a sample size $n$, let $K_n$ be a positive number depending on $n$, whose order of $n$ will be specified below. We construct the desired functions $f_{j, n}$ in two separate cases:

\gap

{\bf Case 1}: $\gamma :=r-1 \in (0,1)$. Let
\[
  \bar{L} \, := \, \min\left( \, \frac{L}{4}, \, \sqrt{\frac{c_0 \gamma  }{12 p_*}} \, \right),
\]
where $p_*>0$ is defined above.
%
%
%
We shall define the functions $f_{j,n}$, $j = 0,1,2,\dots, \lfloor K_n^\gamma \rfloor$ as follows.  First we define the auxiliary functions $\bar{g}_{j,n}$ for $j = 0,1,\dots,\lfloor K_n^\gamma \rfloor$. For $i = 0,1,2,\dots$, let
\[
\bar{g}_{0,n}(x) \, : =  \, \left\{
\begin{array}{ll}
2i\bar{L} K_n^{-\gamma } + \bar{L}K_n^{1-\gamma}(x- \frac{i}{K_n^\gamma}),  & \text{ if } x \in [\frac{i}{K_n^\gamma}, \frac{i}{K_n^\gamma}+\frac{1}{K_n} ) \\
(2i+1)\bar{L}K_n^{-\gamma}, & \text{ if } x \in [\frac{i}{K_n^\gamma}+\frac{1}{K_n}, \frac{i}{K_n^\gamma}+\frac{3}{K_n} )\\
(2i+1)\bar{L}K_n^{-\gamma}  + \bar{L}K_n^{1-\gamma}\big [ x-(\frac{i}{K_n^{\gamma}}+\frac{3}{K_n}) \big],  & \text{ if } x \in [ \frac{i}{K_n^\gamma}+\frac{3}{K_n}, \frac{i}{K_n^\gamma}+\frac{4}{K_n} )\\
2(i+1)\bar{L} K_n^{-\gamma}, & \text{ if } x \in [\frac{i}{K_n^\gamma}+\frac{4}{K_n}, \frac{ i+1}{K_n^\gamma} )
\end{array}\right.
\]
For each $j = 1,2,\dots, \lfloor K_n^\gamma \rfloor$, let $\bar{g}_{j,n} = \bar{g}_{0,n}$ everywhere except on $[\frac{j-1}{K_n^\gamma},\frac{j}{K_n^\gamma})$ on which $\bar{g}_{j,n}$ is defined as follows:
\[
\bar{g}_{j,n}(x) \, := \, \left\{
\begin{array}{ll}
2(j-1)\bar{L}K_n^{-\gamma}, & \text{ if } x \in [\frac{j-1}{K_n^\gamma},\frac{j-1}{K_n^\gamma} + \frac{1}{K_n} )\\
2(j-1)\bar{L} K_n^{-\gamma} + \bar{L}K_n^{1-\gamma}\big[ x - (\frac{j-1}{K_n^\gamma}+\frac{1}{K_n}) \big], & \text{ if } x \in [\frac{j-1}{K_n^\gamma} + \frac{1}{K_n},\frac{j-1}{K_n^\gamma} + \frac{3}{K_n} )\\
2j\bar{L}K_n^{-\gamma},  &   \text{ if } x \in [\frac{j-1}{K_n^\gamma} + \frac{3}{K_n},\frac{j}{K_n^\gamma} )
\end{array}\right.
\]
For each $j = 0,1,2,\dots, \lfloor K_n^\gamma \rfloor$, let $g_{j,n}$ denote the restriction of $\bar{g}_{j,n}$ to $[0,1]$.

\begin{figure}[t]
\begin{tikzpicture}
\draw [<-,thick](0,8)--(0,0)--(3,0);
\draw [ultra thick,dotted](3,0)--(4,0);
\draw [thick] (4,0)--(8,0);
\draw [ultra thick,dotted](8,0)--(9,0);
\draw [thick] (9,0)--(13,0);
\draw [ultra thick,dotted](13,0)--(14,0);
\draw [thick] (14,0)--(15,0);
\draw [thick,dgreen](0,0) -- (0.5,0)--(1.5,2) -- (2,2);
\draw [thick,red](5,2)--(5.5,2)--(6.5,4)--(7,4);
\draw [thick,dbrown](10,4)--(10.5,4)--(11.5,6)--(12,6);
\draw [blue,thick](0,0) -- (0.5,1)--(1.5,1) -- (2,2) -- (3,2);
\draw [blue,ultra thick,dotted](3,2) -- (4,2);
\draw [blue,thick](4,2)--(5,2)--(5.5,3)--(6.5,3)--(7,4)--(8,4);
\draw [blue,ultra thick,dotted](8,4) -- (9,4);
\draw [blue,thick](9,4)--(10,4)--(10.5,5)--(11.5,5)--(12,6)--(13,6);
\draw [blue,ultra thick,dotted](13,6) -- (14,6);
\draw [blue,thick] (14,6)--(15,6);
\draw [thick] (0,-0.1) --(0,0.1);
\draw [thick] (0.5,-0.1) --(0.5,0.1);
\draw [thick] (1,-0.1) --(1,0.1);
\draw [thick] (1.5,-0.1) --(1.5,0.1);
\draw [thick] (2,-0.1) --(2,0.1);
\draw [thick] (5,-0.1) --(5,0.1);
\draw [thick] (5.5,-0.1) --(5.5,0.1);
\draw [thick] (6,-0.1) --(6,0.1);
\draw [thick] (6.5,-0.1) --(6.5,0.1);
\draw [thick] (7,-0.1) --(7,0.1);
\draw [thick] (10,-0.1) --(10,0.1);
\draw [thick] (10.5,-0.1) --(10.5,0.1);
\draw [thick] (11,-0.1) --(11,0.1);
\draw [thick] (11.5,-0.1) --(11.5,0.1);
\draw [thick] (12,-0.1) --(12,0.1);
\draw [thick] (-0.1,1) --(0.1,1);
\draw [thick] (-0.1,2) --(0.1,2);
\draw [thick] (-0.1,3) --(0.1,3);
\draw [thick] (-0.1,4) --(0.1,4);
\draw [thick] (-0.1,5) --(0.1,5);
\draw [thick] (-0.1,6) --(0.1,6);

\node[below] at (0,-.5){0};
\node[below] at (2,-.5){\tiny{$\frac{4}{K_n}$}};
\node[below] at (3.5,-.5){\dots};
\node[below] at (5,-.5){\tiny$\frac{1}{K_n^\gamma}$};
\node[below] at (7,-.5){\tiny{$\frac{1}{K_n^\gamma}+\frac{4}{K_n}$}};
\node[below] at (8.5,-.5){\dots};
\node[below] at (10,-.5){\tiny{$\frac{2}{K_n^\gamma}$}};
\node[below] at (12,-.5){\tiny{$\frac{2}{K_n^\gamma}+\frac{4}{K_n}$}};
\node[left] at (-.5,2){\small$\frac{2\bar{L}}{K_n^\gamma}$};
\node[left] at (-.5,4){\small$\frac{4\bar{L}}{K_n^\gamma}$};
\node[left] at (-.5,6){\small$\frac{6\bar{L}}{K_n^\gamma}$};
\node[right] at (12.3,8.5){\footnotesize Legend};
\draw [thick] (12.3,8.8)--(12.3,6.8)--(14.9,6.8)--(14.9,8.8)--(12.3,8.8);
\node[right] at (12.5,8){\scriptsize$g_{0, n}$};
\draw [thick,blue] (13.2,8)--(14.4,8);
\node[right] at (12.5,7.7){\scriptsize$g_{1, n}$};
\draw [thick,dgreen] (13.2,7.7)--(14.4,7.7);
\node[right] at (12.5,7.4){\scriptsize$g_{2, n} $};
\draw [thick,red] (13.2,7.4)--(14.4,7.4);
\node[right] at (12.5,7.1){\scriptsize$g_{3, n}$};
\draw [thick,dbrown](13.2,7.1)--(14.4,7.1);
\end{tikzpicture}

\caption{Plot of $g_{j,n}$'s near the origin when $\gamma \in (0, 1)$.} \label{fig:basis_plot}

\end{figure}

%

\begin{figure}[t]
\begin{tikzpicture}
\draw [<->,thick](0,9)--(0,0)--(12,0);
\draw [thick,dgreen](0,0) -- (1,0)--(3,3) -- (4,3);
\draw [thick,red](4,3)--(5,3)--(7,6)--(8,6);
\draw [thick,dbrown](8,6)--(9,6)--(11,9)--(12,9);
\draw [blue,thick](0,0) -- (1,1.5)--(3,1.5) -- (4,3);
\draw [blue,thick](4,3)--(5,4.5)--(7,4.5)--(8,6);
\draw [blue,thick](8,6)--(9,7.5)--(11,7.5)--(12,9);
\draw [thick] (0,-0.1) --(0,0.1);
\draw [thick] (1,-0.1) --(1,0.1);
\draw [thick] (2,-0.1) --(2,0.1);
\draw [thick] (3,-0.1) --(3,0.1);
\draw [thick] (4,-0.1) --(4,0.1);
\draw [thick] (5,-0.1) --(5,0.1);
\draw [thick] (6,-0.1) --(6,0.1);
\draw [thick] (7,-0.1) --(7,0.1);
\draw [thick] (8,-0.1) --(8,0.1);
\draw [thick] (9,-0.1) --(9,0.1);
\draw [thick] (10,-0.1) --(10,0.1);
\draw [thick] (11,-0.1) --(11,0.1);

\draw [thick] (-0.1,1.5) --(0.1,1.5);
\draw [thick] (-0.1,3) --(0.1,3);
\draw [thick] (-0.1,4.5) --(0.1,4.5);
\draw [thick] (-0.1,6) --(0.1,6);
\draw [thick] (-0.1,7.5) --(0.1,7.5);

\node[below] at (0,-.5){0};
\node[below] at (4,-.5){$\frac{4}{K_n}$};
\node[below] at (8,-.5){$\frac{8}{K_n}$};
\node[below] at (12,-.5){$\frac{12}{K_n}$};
\node[left] at (-.5,3){$\frac{2\bar{L}}{K_n}$};
\node[left] at (-.5,6){$\frac{4\bar{L}}{K_n}$};
\node[left] at (-.5,9){$\frac{6\bar{L}}{K_n}$};
\node[right] at (12.3,9.5){\footnotesize Legend};
\draw [thick] (12.3,9.8)--(12.3,7.8)--(14.9,7.8)--(14.9,9.8)--(12.3,9.8);
\node[right] at (12.5,9){\scriptsize$g_0$};
\draw [thick,blue] (13.2, 9)--(14.4, 9);
\node[right] at (12.5, 8.7){\scriptsize$g_1$};
\draw [thick,dgreen] (13.2,8.7)--(14.4,8.7);
\node[right] at (12.5,8.4){\scriptsize$g_2$};
\draw [thick,red] (13.2,8.4)--(14.4,8.4);
\node[right] at (12.5, 8.1){\scriptsize$g_3$};
\draw [thick,dbrown](13.2,8.1)--(14.4,8.1);
\end{tikzpicture}

\caption{Plot of $g_{j,n}$'s near the origin when $\gamma = 1$.}
\label{fig:basis_plot2}

\end{figure}

\gap

{\bf Case 2}: $\gamma = 1$.  In this case,  choose
\[
  \bar{L}  \, : = \,  \min\left ( \, L, \, \sqrt{\frac{c_0 }{12 p_*}} \, \right),
\]
and define for $i=0,1,2,\dots$,
\[
\bar{g}_{0,n}(x) \, := \, \left\{
\begin{array}{ll}
\bar{L}\frac{2i}{K_n}+ \bar{L}(x - \frac{4i}{K_n})  & \text{ if } x \in [ \frac{4i}{K_n},\frac{4i+1}{K_n} )\\
\bar{L}\frac{2i+1}{K_n} & \text{ if } x \in [ \frac{4i+1}{K_n},\frac{4i+3}{K_n} )\\
\bar{L}\frac{2i+1}{K_n} + \bar{L}(x - \frac{4i+3}{K_n})& \text{ if } x \in [ \frac{4i+3}{K_n},\frac{4(i+1)}{K_n} )\\
\end{array}\right.
\]
Also define $\bar{g}_{j,n} = \bar{g}_{0,n}$ everywhere except on $[\frac{4(j-1)}{K_n},\frac{4j}{K_n})$, on which
\[
\bar{g}_{j,n}(x) \, := \, \left\{
\begin{array}{ll}
\bar{L}\frac{2(j-1)}{K_n} & \text{ if } x \in [\frac{4(j-1)}{K_n},\frac{4j-3}{K_n} )\\
\bar{L}\frac{2(j-1)}{K_n} + \bar{L}(x - \frac{4j-3}{K_n}) & \text{ if } x \in [\frac{4j-3}{K_n},\frac{4j-1}{K_n} )\\
\bar{L}\frac{2j}{K_n} & \text{ if } x \in [\frac{4j-1}{K_n},\frac{4j}{K_n} )
\end{array}\right.
\]
for $j = 1,2,\dots, \lfloor K_n \rfloor$.  Again, for each $j$, we let $g_{j,n}$ denote the restriction of $\bar{g}_{j,n}$ to $[0,1]$.

%
%

The plots of the functions $g_{j,n}$, $j = 0,1,2,\dots, \lfloor K_n^\gamma \rfloor$ near the origin  constructed above are displayed in Figures~\ref{fig:basis_plot} and \ref{fig:basis_plot2} for Case 1 and Case 2 respectively.  Note that in these plots, $g_{0,n}$ often obstructs the view of other $g_{j,n}$'s, but if $j \ge 1$, then $g_{j,n}$ never obstructs the view of any other function.

Finally, in both the cases,  for each $j = 0,1,2,\dots, \lfloor K_n^\gamma \rfloor$, define
\begin{equation} \label{eqn:f_j}
     f_{j,n}(x) \, := \, \int_0^x g_{j,n}(t)\,dt, \qquad x \in [0, 1].
\end{equation}
We present the following theorem for the above construction, whose proof is given in Section~\ref{subsect:proof_lowbd}.


\begin{theorem} \label{thm:lower_bound}
Consider the function class $\mathcal C_H(r, L)$ with $r\in (1, 2]$, $L>0$, and $\gamma := r-1$. Let $K_n =  \left(\frac{n}{\log n}\right)^\frac{1}{2 r +1}$ and $M_n :=  \lfloor K_n^\gamma \rfloor$.  Then the functions $f_{j, n},  \, j=0, 1, \ldots, M_n$ constructed in (\ref{eqn:f_j})   satisfy  conditions (C1)--(C3). Specifically, for all $n$ sufficiently large,
\begin{enumerate}
  \item[(1)] Each $f_{j,n} \in \mathcal C_H(r, L)$;
  \item[(2)]  For all $j,k \in \{ 0,1,\dots, M_n \}$ with $j \ne k$,  $\|f_{j,n} - f_{k,n}\|_\infty = 2s_n$, where $s_n \asymp ({\log n \over n})^{ {r \over (2r+1) } }$;
  \item[(3)]
   $
      \frac{1}{M_n} \sum_{j=1}^{M_n} K(P_j, P_0)\, \leq \, c_0 \log(M_n).
   $
\end{enumerate}
\end{theorem}

\gap

This theorem, together with the similar argument in \cite[Theorem 2.5]{Tsybakov_book10}, leads to the lower bound of minimax risk of convex estimation in (\ref{eqn:lower_risk}).

\begin{remark}  \rm
  The proposed construction procedure for convex functions on the H\"older class can be extended to other shape constraints, e.g., monotone constraint or shape constraints in terms of  higher order  derivatives.
\end{remark}


%
\subsection{Proof of Theorem~\ref{thm:lower_bound} } \label{subsect:proof_lowbd}

\begin{proof}
We consider the two cases in the proof: $\gamma \in (0, 1)$, and $\gamma=1$.

{\bf Case 1}: $\gamma \in (0, 1)$. For all $n$ (and $K_n$) sufficiently large,
the following properties of $g_{j, n}$'s can be easily verified with the help of Figure~\ref{fig:basis_plot}: for any $x, y \in [0, 1]$, suppose that
\begin{itemize}
  \item [(i)] $0 < |x-y| \leq \frac{4}{K_n}$.  Then
     \[
       \max_j \frac{| g_{j,n}(x)-g_{j,n}(y)|}{|x - y|} \, \le \,   \frac{| g_{1,n}(x)-g_{1,n}(y)|}{|x - y|} \Big|_{x=K^{-1}_n, \, y= 2 K^{-1}_n} \, \le \, \bar L K^{1-\gamma}_n.
     \]
  \item [(ii)] $\frac{4}{K_n} < |x-y| \le \frac{1}{K_n^\gamma}$.  Then
     \[
        \max_j  | g_{j,n}(x)-g_{j,n}(y)| \, \le \,  | g_{0,n}(x)-g_{0,n}(y)|\big|_{x=0, \, y= 4 K^{-1}_n} \, \le \, 2 \bar L K^{-\gamma}_n.
     \]
  \item [(iii)]   $\frac{1}{K_n^\gamma} <|x-y| \le 1 $. Without loss of generality, let $x<y$ with $y= r K^{-\gamma}_n + s(x, y)$ for some $r\in \mathbb N$ and $0\le s(x, y)<  K^{-\gamma}_n$.
It can be shown that
  \begin{eqnarray*}
     \max_{j} \frac{| g_{j,n}(x)-g_{j,n}(y)|}{|x - y|} &  \le & \frac{| g_{1,n}(x)-g_{1,n}(y)|}{|x - y|}\Big|_{x=K^{-1}_n, \ y=r K^{-\gamma}_n + 3 K^{-1}_n} \\
  & \le &  { 2(r+1) \bar L K^{-\gamma}_n \over  r K^{-\gamma}_n + 2 K^{-1}_n } \, \le \, { 2(r+1) \bar L K^{-\gamma}_n \over  r K^{-\gamma}_n  } \, \le \, 4 \bar L \, \le \, L.
   \end{eqnarray*}
 \end{itemize}

Along with these properties,  we show the three conditions as follows:

(1)
Obviously, each $g_{j,n}$ is nondecreasing, and hence each $f_{j,n}$ is convex.  Furthermore, to show that each function $f_{j,n}$ is in the H\"older class $H^r_L$, we consider the following three cases:
\begin{itemize}
   \item [(1.1)] $0 < |x-y| \leq \frac{4}{K_n}$.
Then, by (i), we have
\[
\frac{|f^\prime_{j,n}(x)-f^\prime_{j,n}(y)|}{|x - y|^\gamma} \, = \, \frac{|f^\prime_{j,n}(x)-f^\prime_{j,n}(y)|}{|x - y|}|x-y|^{1-\gamma} \, \leq \, \bar{L}K_n^{1-\gamma}\left(\frac{4}{K_n}\right)^{1-\gamma} \, \leq \, 4 \bar L \, \le \, L.
\]
  \item [(1.2)] $\frac{4}{K_n} < |x-y| \leq \frac{1}{K_n^\gamma}$.
Then, by (ii), we have
\[
\frac{|f^\prime_{j,n}(x)-f^\prime_{j,n}(y)|}{|x - y|^\gamma} \, \leq \, \frac{2 \bar L K^{-\gamma}_n }{  |x - y|^\gamma } \, \le \,
 2\bar{L}K_n^{-\gamma}\left(\frac{4}{K_n}\right)^{-\gamma}  \, \leq \, 2 \bar L \, \le \, L.
\]
 \item [(1.3)] $\frac{1}{K_n^\gamma} < |x-y| \le 1$. By (iii), we obtain
\[
\frac{|f^\prime_{j,n}(x)-f^\prime_{j,n}(y)|}{|x - y|^\gamma} \, \leq \, \frac{|f^\prime_{j,n}(x)-f^\prime_{j,n}(y)|}{|x - y|} \, \leq \, L.
\]
\end{itemize}
This shows that condition (1) holds.

(2)  Let $j,k \in \{ 0,1,\dots, M_n\}$ with $j < k$ without loss of generality. It follows from the definitions of $g_{j, n}$ and $f_{j, n}$ and Figure~\ref{fig:basis_plot} that
\begin{itemize}
  \item [(2.1)] if $j=0$, then $f'_{j, n}(x) = f'_{k, n}(x)$ for all $x\in [0, 1]$ except on the set $\mathcal S_{0k}:= ( (k-1) K^{-\gamma}_n, (k-1) K^{-\gamma}_n + 2/K_n) \cup ( (k-1) K^{-\gamma}_n + 2/K_n, (k-1) K^{-\gamma}_n + 4/K_n )$;
  \item  [(2.2)] if $j \ge 1$, then  $f'_{j, n}(x) = f'_{k, n}(x)$ for all $x\in [0, 1]$ except on  the set $\mathcal S_{jk} := ( (j-1) K^{-\gamma}_n, (j-1) K^{-\gamma}_n + 2/K_n) \cup ( (j-1) K^{-\gamma}_n + 2/K_n, (j-1) K^{-\gamma}_n + 4/K_n ) \cup ( (k-1) K^{-\gamma}_n, (k-1) K^{-\gamma}_n + 2/K_n) \cup ( (k-1) K^{-\gamma}_n + 2/K_n, (k-1) K^{-\gamma}_n + 4/K_n )$.
\end{itemize}
   Hence,  the set of critical points of $f_{j, n} - f_{k, n}$ is $[0, 1]\setminus \mathcal S_{jk}$. Furthermore, in view of piecewise linearity of $g_{j, n}$'s, it is easy to see that
\begin{itemize}
  \item [(a)] 
    $f'_{j, n}(x) - f'_{k, n}(x)= g_{j, n}(x)- g_{k, n}(x)>0$ for all $x\in \big( (k-1) K^{-\gamma}_n + 1/K_n, (k-1) K^{-\gamma}_n + 2/K_n \big)$, and $f'_{j, n}(x) - f'_{k, n}(x)<0$ for all $x\in  \big ( (k-1) K^{-\gamma}_n + 2/K_n, (k-1) K^{-\gamma}_n + 3/K_n  \big)$;
  \item [(b)]  for case (2.2),   $f'_{j, n}(x) - f'_{k, n}(x)<0$ for all $x\in  \big ( (j-1) K^{-\gamma}_n + 1/K_n, (j-1) K^{-\gamma}_n + 2/K_n  \big)$, and $f'_{j, n}(x) - f'_{k, n}(x)>0$ for all $x\in  \big ( (j-1) K^{-\gamma}_n + 2/K_n, (j-1) K^{-\gamma}_n + 3/K_n  \big)$;
  \item [(c)] $g'_{j,n}(x)=g'_{k,n}(x)=0$ for all $x \in [0, 1]\setminus \mathcal S_{jk}$ except $ x= (k-1) K^{-\gamma}_n + 2/K_n$ and $x=  (j-1) K^{-\gamma}_n + 2/K_n$ (if $j \ge 1$).
\end{itemize}
Moreover, $f_{j, n}(x) = f_{k, n}(x)$ for $x = 0, 1$.
This shows that  $ |  f_{j, n} (x)- f_{k, n}(x) |$  achieves a local maximum at $x^*=  (k-1) K^{-\gamma}_n + 2/K_n$ and/or $z^*:=  (j-1) K^{-\gamma}_n + 2/K_n$ (the latter holds only if $j \ge 1$). Due to the symmetry of non-overlapping regions of $g_{j, n}$ and $g_{k, n}$,
we have $\| f_{j, n} - f_{k, n}\|_\infty = |  f_{j, n} (x^*)- f_{k, n}(x^*) | = |  f_{j, n} (z^*)- f_{k, n}(z^*) |$. Furthermore, it can be verified that $f_{j, n}(x)= f_{k, n}(x)$  at $x =(k-1)K_n^{-\gamma}$.  
%
%
%
Therefore,
\begin{eqnarray*}
 \|f_{j,n} - f_{k,n}\|_\infty &  = & \left|f_{j,n}\Big((k-1)K_n^{-\gamma} + \frac{2}{K_n}\Big) - f_{k,n}\Big((k-1)K_n^{-\gamma} + \frac{2}{K_n}\Big)\right| \\
			     &  =  &  \left|(f_{j,n}- f_{k,n})\Big((k-1)K_n^{-\gamma} + \frac{2}{K_n}\Big) - (f_{j,n}- f_{k,n})\Big((k-1)K_n^{-\gamma} \Big)\right| \\
			     &  = &  \int_{(k-1)K_n^{-\gamma}}^{(k-1)K_n^{-\gamma} + \frac{2}{K_n}}   \Big(g_{j,n}(t) - g_{k,n}(t) \Big)  \,dt \\
			     &  = &  2 \int_{(k-1)K_n^{-\gamma}}^{(k-1)K_n^{-\gamma} + \frac{1}{K_n}} \bar{L} K_n^{1-\gamma} \Big( t-(k-1)K_n^{-\gamma}\Big) \,dt \\	
			     &  =  & 2\int_0^{\frac{1}{K_n}} \bar{L} K_n^{1-\gamma}t   \,dt \, = \, \bar{L} K_n^{-(1+\gamma)}   \\
			     &  = &  \bar{L} K_n^{-r} =  2s_n,
\end{eqnarray*}
where $s_n := \bar L K^{-r}_n/2 = \bar L/2 \cdot \big( { \log n \over n} \big)^{ r \over 2 r +1}$,  and thus condition (2) holds.

(3)  To show this condition, we first collect a few results about $f_{j, n}$'s to be used later:
\begin{itemize}
  \item [(3.1)]  For each $j=1, 2, \ldots, M_n$, $\displaystyle \int^1_0 \big( f_{j, n}(x) - f_{0, n}(x) \big)^2 dx =  \int^1_0 \big( f_{1, n}(x) - f_{0, n}(x) \big)^2 dx$.
  \item [(3.2)]  Let $h_j :=  (f_{j, n}- f_{0, n})^2$ for $j=1, \ldots, M_n$.
  Since  $X_i=i/n$, it follows from analysis of numerical integration and condition (2) that for each $j$,
  \[
      \left | \int^1_0 h_j(x) d x - { \sum^n_{i=1}  h_j(X_i) \over n} \right| \, \le \, \frac{ \max_{x\in [0, 1]} | h'_j(x)|  }{2 n }  \, \le \,  \frac{  \| g_{j, n} -  g_{0, n} \|_\infty \cdot \| f_{j, n} - f_{0, n}\|_\infty  }{n }  \le \frac{ 2 \bar L^2}{n K^{1+2\gamma}_n}.
  \]
 \item [(3.3)]  The following holds:
  \begin{align*}
    \int_0^1  (f_{1, n} (x) - f_{0, n}(x))^2 \, dx  & =  2 \int_0^\frac{2}{K_n}  (f_{1, n}(x) - f_{0,n} (x))^2 \, dx =  2 \int_0^\frac{1}{K_n} \left( \frac{\bar{L} K_n^{1-\gamma} x^2}{2}\right)^2 \, dx  \\
     &   
     \quad
     +  2  \int_\frac{1}{K_n}^\frac{2}{K_n} \left(\frac{\bar{L} K_n^{-(1+\gamma)}}{2} + \bar{L}K_n^{-\gamma}\left(x -\frac{1}{K_n}\right) -\frac{\bar{L} K_n^{1-\gamma}}{2}\left(x- \frac{1}{K_n}\right)^2 \right)^2  \, dx \\
     & =  2 \bar L^2 K^{-2\gamma -3}_n \left(\frac{1}{20} + \frac{43}{60} \right) =  \frac{ 2 \bar L^2} {K^{2 r+1}_n} \cdot {46 \over 60 } .
   \end{align*}
%
%
\end{itemize}

In light of the above results, we have  for each $j=1, \ldots, M_n$,
\begin{eqnarray*}
 K(P_j,P_0) & \le & p_* \sum_{i=1}^{n} (f_{j, n}(X_i) - f_{0, n}(X_i))^2 \, \le \,  p_* \left(n \int_0^1  (f_{j, n}(x) - f_{0, n}(x))^2 \, dx + \frac{ 2 \bar L^2}{ K^{1+2\gamma}_n} \right) \\
	    & \le &  p_*  \left( n \int_0^1  (f_{1, n}(x) - f_{0, n}(x))^2 \, dx + \frac{ 2 \bar L^2}{ K^{1+2\gamma}_n}  \right) \\
	    & \le &  p_*  \left( \frac{2 n \bar L^2 }{K^{2 r+1}_n} \cdot {46 \over 60 }  + \frac{ 2 \bar L^2}{ K^{1+2\gamma}_n}  \right)  \\
	    & \le &  \frac{c_0 \gamma}{6}\log(n)
\end{eqnarray*}
for all $n$ sufficiently large, where the last inequality follows from the definition of $\bar L$ and the order of $K_n$. Consequently,
\[
   \frac{1}{M_n} \sum_{j=1}^{M_n} K(P_j,P_0) \, \leq \,  \frac{c_0 \gamma}{6}\log(n).
\]
Finally, since $\gamma \in (0, 1)$, we have, for all $n$ sufficiently large,
\[
  \log M_n  \ge  0.9 \gamma \log K_n = 0.9 \gamma \log\left(\left(\frac{n}{\log n}\right)^\frac{1}{2r+1}\right) = \left(\frac{0.9 \gamma}{2\gamma+3}\right)\log \Big(\frac{ n}{
  \log n} \Big) \geq  \frac{\gamma}{6}\log n.
\]
This establishes condition (3) and hence completes the proof for Case 1.

\gap

{\bf Case 2}: $\gamma=1$.  We show the three conditions in a similar manner as in Case 1:

(1) Clearly, each $f_{j, n}$ is convex on $[0, 1]$. Further, it is easy to show via the definition of $g_{j, n}$ and Figure~\ref{fig:basis_plot2} that  for any $0\le x < y \le 1$, $\frac{| g_{j,n}(x)-g_{j,n}(y)|}{|x - y|} \leq \bar{L} \leq L$ for each $j=0, 1, \ldots, \lfloor K_n \rfloor$.  This thus implies that each $f_{j, n} \in \mathcal C_H(r, L)$, leading to condition (1).

(2) Let $0\le j < k \le M_n:=\lfloor K_n \rfloor $.  It follows from a similar argument as in (2) of Case 1 that
$\| f_{j, n} - f_{k, n}\|_\infty$ is achieved at $x^*= \big(4 (k-1) +2 \big)/K_n$ and $f_{j, n}(x)= f_{k, n}(x)$  at $x= 4(k-1)/K_n$. 
Therefore, we have
\begin{eqnarray*}
 \|f_{j,n} - f_{k,n} \|_\infty &  = & \left|f_{j,n}\left(\frac{4(k-1)+2}{K_n}\right) - f_{k,n}\left(\frac{4(k-1)+2}{K_n}\right)\right | \\
			     &  =  &  \left|(f_{j,n}- f_{k,n})\left( \frac{4(k-1)+2}{K_n} \right) - (f_{j,n}- f_{k,n})\left(\frac{4(k-1)}{K_n}\right)\right| \\
			     &  =  & \int_{\frac{4(k-1)}{K_n}}^{\frac{4(k-1)+2}{K_n}} \left(g_{j,n}(t) - g_{k,n}(t)\right)  \,dt \\
			     &  =  & 2\int_{\frac{4(k-1)}{K_n}}^{\frac{4(k-1)+1}{K_n}} \bar{L} \left( t-\frac{4(k-1)}{K_n}\right) \,dt\\	
			     &  =  & 2\int_0^{\frac{1}{K_n}} \bar{L} t  \,dt \,  = \,   \bar{L} K_n^{-2}  \, = \, 2s_n,
\end{eqnarray*}
where $s_n := \bar L K^{-2}_n/2 = \bar L/2 \cdot \big( { \log n \over n} \big)^{ 2 \over 5}$ (for $r=2$),  and thus condition (2) holds for $\gamma = 1$.

(3)  First of all, it is easy to see that the conditions in (3.1) and (3.2) in Case 1 remain valid for $\gamma=1$.  To show the condition in (3.3) for $\gamma=1$, we  have
\begin{align*}
    \int_0^1  \big(f_{1, n} (x) - f_{0, n}(x) \big)^2 \, dx  & =  2 \int_0^\frac{2}{K_n} \big (f_{1, n}(x) - f_{0,n} (x) \big)^2 \, dx =  2 \int_0^\frac{1}{K_n} \left( \frac{\bar{L}  x^2}{2}\right)^2 \, dx  \\
     &   
     \quad
     +  2  \int_\frac{1}{K_n}^\frac{2}{K_n} \left(\frac{\bar{L} K_n^{-2}}{2} + \bar{L}K_n^{-1}\left(x -\frac{1}{K_n}\right) -\frac{\bar{L} }{2}\left(x- \frac{1}{K_n}\right)^2 \right)^2  \, dx \\
     & =  2 \bar L^2 K^{-5}_n \left(\frac{1}{20} + \frac{43}{60} \right).
   \end{align*}
  By using these results and a similar argument as in (3) of Case 1, we have  for all $n$ sufficiently large, $ K(P_j,P_0) \le  \frac{c_0}{6}\log(n)$ for each $j=1, \ldots, M_n$ such that
\[
   \frac{1}{M_n} \sum_{j=1}^{M_n} K(P_j,P_0) \, \leq \,  \frac{c_0}{6}\log(n).
\]
Again, in view of
\[
  \log M_n  \ge  0.9  \log K_n = 0.9  \log\left(\left(\frac{n}{\log n}\right)^\frac{1}{5}\right) = \left(\frac{0.9}{5}\right)\log \Big(\frac{ n}{
  \log n} \Big) \geq  \frac{1}{6}\log n
\]
for all $n$ sufficiently large, we obtain condition (3). This completes the proof for Case 2.
\end{proof}


%
%
%
%
%

\section{Convex $P$-spline Estimator and Minimax Upper Bound of Convex Estimators} \label{sect:upper_bound}

In this section, we develop a convex  $P$-spline estimator subject to the second order difference penalty that achieves the optimal rate of convergence.  We first formulate the convex $P$-spline estimator as a constrained quadratic optimization problem, and establish optimality conditions for spline coefficients (cf. Section~\ref{sec:formulation}). We then develop a critical uniform Lipschitz property in the infinity norm for the optimal spline coefficients  (cf. Theorem~\ref{thm:uniform_Lipschitz}). Equipped with these results,  we establish uniform convergence of the convex $P$-spline estimator, and this leads to the minimax upper bound under the sup-norm in Section~\ref{sect:rate_optimal}.

%
\subsection{Convex $P$-spline Estimator and Optimal Spline Coefficients} 
\label{sec:formulation}

Consider a $P$-spline estimator for the convex estimation problem defined in  (\ref{eqn:convex_estimation}).
%
%
Specificially, let $\big\{ B^{[p]}_k : k = 1, \ldots, K_n + p \,
\big\}$ be the $p\,$th degree B-spline basis with knots $0 =
\kappa_0 < \kappa_1 < \cdots < \kappa_{K_n} = 1$ and extension to
$\kappa_{-p}<\kappa_{-p+1}<\cdots <\kappa_0$ and $\kappa_{K_n}< \kappa_{K_n+1}<\cdots <\kappa_{K_n+p}$ on the boundary. For simplicity, consider equally
spaced knots, i.e.,   $\kappa_k = k/K_n$ with $k=-p, \ldots, K_n+p$, where the support of each basis function is $[(k-p-1)/K_n, k/K_n]\cap [0, 1]$, $k=1, \ldots, K_n+p$.
The value of $K_n$ will depend upon $n$ as shown below, and we also assume
$n/K_n$ to be an integer denoted by $M_n$.  In what follows, we consider the penalized linear B-spline  based convex estimator, namely, $p=1$.

Let  $\Delta$ be the backward difference operator, i.e., $\Delta (b_k) := b_k - b_{k-1}$ and $\Delta^m (b_k) = \Delta (\Delta^{m-1}(b_k))$ with $m\in \mathbb{N}$, and consider the polyhedral cone 
\[
   \Omega \, = \, \{b \in \mathbb R^{K_n+1} \, | \, \Delta^2 (b_k)  = b_k-2b_{k+1}+b_{k+2}\ge 0, \ k=1, \ldots, K_n-1 \}.
\]
The  constrained optimization problem for $P$-spline coefficients with the penalty on  the $m$th order difference is 
\begin{equation}\label{equ:bm}
  \hat b^{[m]} \, : = \, \arg\min_{b\in \Omega} \sum_{k=1}^{K_n+1}   \Big( y_i - \sum_  {k=1}^{K_n+1}b_k B_{k}^{[1]}(x_i)\Big)^2     +
  \lambda^*\sum_{k=m+1}^{K_n+1}\big(\Delta^m b_k\big)^2,
\end{equation}
where $\hat b^{[m]} \in \mathbb R^{K_n+1}$, and $\lambda^*>0$ is the penalty parameter dependent on $n$ whose order will be determined later.
Then the convex $P$-spline estimator with $p=1$ is given by:
\[
  \hat f^{[m]} (x)  \, = \, \sum_{k=1}^{K_n+1}\hat b^{[m]}_k B^{[1]}_k(x).
\]
Since the knots are equally spaced, it is easy to see that if the B-spline coefficient
vector $\hat b^{[m]}$ is in $\Omega$, then $\hat f^{[m]}$ is convex.  We consider the second order difference penalty in this paper, i.e., $m=2$.

In order to establish the optimality conditions for the optimal spline coefficient $\hat b^{[m]}$, we introduce more notation.  Denote the $n\times (K_n+1)$ design matrix by $X = \big[B^{[1]}_k(x_j) \big]_{j, k}$ and  $\beta_n := \sum_{i=1}^n \big(
B_k^{[1]}(x_i) \big)^2$ for $k=2, \ldots, K_n$. Since the knots are equally spaced,
$\beta_n$ is independent of $k=2, \ldots, K_n$.  For the linear B-spline,
$\big( \beta_n \frac{K_n}{n}  \big)$ converges to a positive constant as $(n/K_n) \rightarrow \infty$. Hence, there exists a positive constant $C_{\beta}$ such that
\begin{equation} \label{eqn:beta_n}
      \beta_n \ge  C_{\beta} \cdot \frac{n}{K_n}, \qquad \forall \ n, K_n.
\end{equation}

Moreover, define $ \Gamma := X^T X/\beta_n \in \mathbb R^{(K_n+1)\times (K_n+1)}$.  It is easy to show that $\Gamma$ is positive definite and tridiagonal.
Let $y=(y_1, \ldots, y_n)^T$ and define the weighted response vector $\bar y := X^T y/\beta_n \in \mathbb R^{K_n+1}$.
Furthermore, let
%
%
$D_2 \in \mathbb{R}^{(K_n-1)\times (K_n+1)}$ be the second-order difference matrix,  i.e.,
\[
    D_2 \, = \, \left[\begin{array}{ccccccccc}
   1 & -2 & 1  &0& \cdots & 0&  0 & 0 & 0\\
   0 & 1 & -2 &1 & \cdots & 0&  0 & 0 & 0\\
    & \cdots & & & \cdots &  & \cdots & \cdots \\
   0 & 0 & 0 & 0 & \cdots &  1& -2 & 1 &0 \\
   0 & 0 &  0 &0 &\cdots &  0& 1 & -2 & 1
   \end{array} \right].
\]
Then the convex constraint on spline coefficients is defined by the polyhedral
cone $\Omega:= \big\{b \in \mathbb R^{K_n+1}: D_2 b \ge 0 \big\}$.
Define $\Lambda := \Gamma+ \lambda D^T_2 D_2$, where $\lambda:=\lambda^*/\beta_n>0$. Therefore, the underlying optimization problem (\ref{equ:bm}) with $p=1$ and $m=2$ becomes the following quadratic program
\begin{equation} \label{eqn:bm2}
       \hat b \, = \, \arg\min_{b\in \Omega} \, {1\over 2}\, b^T \Lambda \, b -   b^T\bar{y},
\end{equation}
where we drop the superscript in $\hat b^{[m]}$ for notational convenience.
 In what follows, we treat the optimal spline coefficient vector $\hat b:\mathbb R^{K_n+1}\rightarrow \mathbb R^{K_n+1}$ as a function of  $\bar y$. It is known that the function $\hat b$ is piecewise linear and Lipschitz continuous \cite{FPang_book03}.
 To obtain a piecewise linear formulation of $\hat b$, consider the optimality conditions in the KKT form:
\begin{equation} \label{eqn:b_KKT}
       \Lambda \hat b - \bar y - D^T_2 \chi \,  = \, 0, \qquad \quad    0  \, \le \, \chi \, \perp \, D_2 \hat b \, \ge \, 0,
\end{equation}
where $\chi \in \mathbb R^{K_n-1}$ is the Lagrange multiplier, and $u \perp v$ means that two vectors $u, v$ are orthogonal, i.e., $u^T v=0$.  Here  the latter condition in  (\ref{eqn:b_KKT}) is known as the complementarity condition \cite{CPStone_book92} in constrained optimization.
Therefore,  linear selection functions (or linear pieces) of $\hat b$ can be determined by index sets $\alpha = \{ \, i \, | \, (D_2 \hat b)_i =0 \}\subseteq \{1, \ldots, K_n-1 \}$, where $\alpha$ may be empty. For a given index set $\alpha$, the optimality conditions in  (\ref{eqn:b_KKT}) yield the equations
\begin{equation} \label{eqn:D2linear_piece}
   (D_2)_{\alpha \bullet} \hat b =0, \qquad \chi_{\ol\alpha} =0,  \qquad \Lambda \hat b - \bar y - \big( (D_2)_{\alpha\bullet} \big)^T \chi_\alpha = 0,
\end{equation}
where $(D_2)_{\alpha\bullet}$ denotes the rows of $D_2$ indexed by the index set $\alpha$, and
$\ol\alpha$ is the complement of $\alpha$, namely, $\ol\alpha := \{ 1, \ldots, K_n-1 \} \setminus \alpha$. Let $\hat b^\alpha$ denote a linear selection function obtained from (\ref{eqn:D2linear_piece}) corresponding to the index set $\alpha$, and let $(F_\alpha)^T$ be a matrix formed by the basis of the null space of $(D_2)_{\alpha\bullet}$. (When $\alpha$ is the empty set, $(F_\alpha)^T$ will be the identity matrix.) In view of the first equation in (\ref{eqn:D2linear_piece}), $\hat b^\alpha$ takes the form $(F_\alpha)^T \wt b^\alpha$, where the vector $\wt b^\alpha$ consists of the free variables of the equation $(D_2)_{\alpha \bullet} \hat b =0$.
 It is easy to show via (\ref{eqn:D2linear_piece}) and a standard argument that $F_\alpha \Lambda F^T_\alpha \wt b^\alpha = F_\alpha \bar y$. Since $F_\alpha$ is row linearly independent and $\Lambda$ is positive definite, we obtain the linear selection function corresponding to $\alpha$ as
\[
     \hat b^\alpha (\bar y)  \, = \, F^T_\alpha \big( \, F_\alpha \Lambda F^T_\alpha \, \big)^{-1} F_\alpha \bar y.
 \]
Hence, for a fixed $K_n$, $\hat b$ has $2^{K_n-1}$ linear selection
functions. Note that there are multiple ways to construct (equivalent) piecewise linear formulation of $\hat b$, but the above approach offers flexibility to select suitable $F_\alpha$ for analytic properties of $\hat b$ in the next section.
%

%
\subsection{Uniform Lipschitz Property of Optimal Spline Coefficients} 
\label{sec:uniform_Lips}

The constrained optimization problem (\ref{eqn:bm2}) does not admit a closed form solution for the optimal spline coefficients, and this poses great difficulties in asymptotical statistical analysis.  In this section, we characterize a critical property for the optimal spline solution $\hat b$ (with $m=2$) pertaining to the uniform Lipschitz constant, regardless of $K_n$, $\alpha$, and $\lambda$, when the $\ell_\infty$-norm is used (cf. Theorem~\ref{thm:uniform_Lipschitz}).
Toward this end, we discuss more about the piecewise linear formulation of $\hatb$.

For the given design matrix $X$, define
$\alpha_n : =  \sum^n_{i=1} \big( B^{[1]}_k(x_i) \big)^2$ for $k=1, K_n+1$, and $\gamma_n:=\sum^n_{i=1} B^{[1]}_k(x_i) B^{[1]}_{k+1}(x_i)$ for $k=2, \ldots, K_n$, where $\gamma_n$ is independent of $k=2, \ldots, K_n$ due to the equally spaced knots. Moreover, let  $\theta_n := \alpha_n/\beta_n$, and $\eta_n :=\gamma_n/\beta_n$, where $\beta_n$ is defined before (\ref{eqn:beta_n}).  It is easily verified via the B-spline properties that $\theta_n \rightarrow 1/2$ and $\eta_n \rightarrow 1/4$ as $n/K_n \rightarrow \infty$. Note that the tridiagonal matrix $\Gamma$ is positive definite and is given by
\begin{equation}  \notag 
 \Gamma \, = \,  \frac{X^T X }{\beta_n} \, = \, \left [\begin{array}{cccccc}
   \theta_n & \eta_n & 0 & 0 & \cdots & 0 \\
    \eta_n & 1&  \eta_n & 0 &  \cdots & 0 \\
   & \ddots & \ddots & \ddots &\ddots &\\
     & &  \eta_n & 1 & \eta_n & 0 \\
   & & & \eta_n & 1 & \eta_n \\
    0&0 & \cdots & 0 & \eta_n & \theta_n
  \end{array}\right ]  \in \mathbb R^{(K_n+1) \times (K_n+1)}.
\end{equation}
such that the positive definite matrix
$
   \Lambda \, = \,  \Gamma+ \lambda D^T_2 D_2 \in \mathbb R^{(K_n+1) \times (K_n+1)}.
$

It follows from complementarity theory  \cite{CPStone_book92, FPang_book03, SWang_SICON11} that the optimal solution $\hat b$ is a piecewise linear and Lipchitz continuous function of $\bar y$
determined by an index set $\alpha = \{ \, i \, | \, (D_2 \hat b)_i
=0 \} \subseteq \{1, \ldots, K_n-1 \}$%
($\alpha$ may be empty). Specifically, for given $\hat b$ and
$\alpha$, we define a vector $\wt b^\alpha$ and an associated
family of index sets $\{ \beta^\alpha_i \}$ as follows.
Let the index set  $\vartheta :=\{ j-1 \, | \,  (\Delta^2(\hat b_j) =0,  j \in \{3, 4, \ldots, K_n+1\} \}$.    Since  $\hat b_i= (\hat b_{i-1} + \hat b_{i+1})/2$ for each $i \in \vartheta$,  it is easy to see that  $\hat b_i, i \in \vartheta$ are the basic variables of the equation $(D_2)_{\alpha\bullet}\hat b=0$, while $\hat b_i, i \not\in \vartheta$ are the free variables. Let $1=i_1<i_2< \cdots < i_\ell=K_n+1$ be the elements of $\ol\vartheta$,  i.e., the complement of $\vartheta$, where $\ell :=K_n+1 - | \alpha|$. Let $\wt b^\alpha := (\hat b_{i_1}, \ldots, \hat b_{i_\ell})^T$ be the vector of free variables.

A specific basis for the null space of $(D_2)_{\alpha\bullet} $ is constructed via the linear B-splines $g_1, g_2, \ldots, g_\ell$ on the interval $[1, K_n+1]$ with nodes $1=i_1<i_2< \cdots <i_\ell=K_n+1$ recently introduced in \cite{WangShen_SICONConvex12}. To be self-contained, we present its construction as follows. Let $I_S$ denote the indicator function of a set $S$, and consider
\begin{equation} \label{eqn:lin_spline_1}
   r_1(t) = \frac{i_2- t}{ i_2 - i_1} \cdot I_{[i_1, i_2]}, \qquad r_\ell(t) = \frac{t- i_{\ell-1}}{i_\ell- i_{\ell-1} } \cdot I_{[i_{\ell-1}, i_\ell]},
\end{equation}
and for each $s=2, \ldots, \ell-1$,
\begin{equation} \label{eqn:lin_spline_2}
    r_s(t) = \frac{t- i_{s-1} }{ i_s- i_{s-1} }  \cdot I_{ [i_{s-1}, i_s)} + \frac{i_{s+1} - t }{ i_{s+1} - i_s} \cdot I_{[i_s, i_{s+1}]}.
\end{equation}
For each $s \in\{1, \ldots, \ell\}$, let the vector $v_s:=\big( r_s(1), r_s(2), \ldots, r_s(K_n+1) \big)^T$. It is easy to show that $v_1, \ldots, v_\ell$ are linearly independent.
Moreover, by the piecewise linear property of $r_s$,  we see that each vector $v_s$ is in the null space of $(D_2)_{\alpha\bullet} $. %
Since the null space of $(D_2)_{\alpha\bullet} $ is of dimension $\ell$ and $v_1, \ldots, v_\ell$ are linearly independent,  $\{v_1, \ldots, v_\ell\}$ forms a basis of the null space of $(D_2)_{\alpha\bullet} $. Letting $(F_\alpha)^T :=(v_1, \ldots, v_\ell)$, we have $\hat b^\alpha = (F_\alpha)^T \wt b^\alpha$.

Moreover, let $\ell_0:=1$ and $\ell_1$ be the smallest element $i_s>1$ in $\ol\vartheta$ such that $i_{s+1}-i_s=1$, and $\beta^\alpha_1:=\{ 1, 2, \ldots, \ell_1-1, \ell_1\}$. Next let $\ell_2$ be the smallest element $i_s>\ell_1$ in $\ol\vartheta$ such that $i_{s+1}-i_s=1$, and
$\beta^\alpha_2:=\{ \ell_1+1, \ell_1+2, \ldots, \ell_2\}$. Continuing this process, we obtain $\beta^\alpha_k$'s untill we reach $i_\ell=K_n+1$. If $i_{\ell-1}=K_n$, we also choose the last $\beta^\alpha_k$ as $\{ K_n+ 1 \}$. Suppose that the above process leads to the sets $\beta^\alpha_1, \ldots, \beta^\alpha_L$.
It is clear that $\{ \beta^\alpha_i \}^L_{i=1}$ forms a disjoint partition of $\{ 1, \ldots,
K_n+1\}$, namely, $ \bigcup^L_{i=1} \, \beta^\alpha_{i}=\{ 1,
\ldots, K_n+1\}$ and $\beta^\alpha_{j} \cap \beta^\alpha_{k} =
\emptyset$ whenever $j \ne k$.
Let $m^\alpha_i:=|\beta^\alpha_i|$,
   where $i=1, \ldots, L$.
   Note that    if $m^\alpha_i > 1$, then $m^\alpha_i \ge 3$.
Recall $\ell:=K_n+1- |\alpha|$.
   It follows from the definition of $\beta^\alpha_i$ and the construction of $F_\alpha$ that
\begin{equation} \label{eqn:F_alp}
     F_\alpha
   = \begin{bmatrix} F_{\alpha, 1} &  & & \\ & F_{\alpha, 2} & & \\ & & \ddots & \\& & & F_{\alpha, L}
   \end{bmatrix} \in \mathbb R^{\ell \times (K_n+1) },
   \end{equation}
    where each matrix
   block corresponding to $\beta^\alpha_k$ is shown as follows: if
   $m^\alpha_k=1$, then $F_{\alpha, k}  = 1$; otherwise, assume that the nodes in $\ol\vartheta$ on $[\ell_{k-1}, \ell_k]$ are $i_{s'}=\ell_{k-1} < i_{s'+1}< \cdots <=i_{s'+w_k}=\ell_k$ for some $w_k$. Let $h^\alpha_{k, j} := i_{s'+j} - i_{s'+j-1} \ge 2$ for each $j=1, \ldots, w_k$.
Then $F_{\alpha, k}\in \mathbb R^{ (w_k+1) \times m^\alpha_k}$ constructed  from the linear splines (\ref{eqn:lin_spline_1})-(\ref{eqn:lin_spline_2}) corresponding to $\beta^\alpha_k$ is given by
  \begin{eqnarray}
     \lefteqn{  F_{\alpha, k} = }  \label{eqn:F_k} \\
      &  {\tiny \left[ \begin{array}{cccccccccccccccccccccccccc}
      1 & \frac{h^\alpha_{k,1} -1}{h^\alpha_{k,1} }
     &
       \frac{h^\alpha_{k,1}-2}{h^\alpha_{k,1}} & \cdots &  \frac{1}{h^\alpha_{k, 1} } & 0 &  & \cdots &  & &  \\
       0 & \frac{1}{h^\alpha_{k, 1}} & \frac{2}{h^\alpha_{k, 1} } & \cdots &
        \frac{h^\alpha_{k,1} -1}{h^\alpha_{k,1}  } & 1 & \frac{h^\alpha_{k,2} -1}{h^\alpha_{k,2}}
        & \cdots & \frac{1}{h^\alpha_{k, 2} } & 0  &  0 &  & \cdots & \\
        0 &  & \cdots & & & 0 & \frac{1}{h^\alpha_{k, 2} } & \cdots & \frac{h^\alpha_{k,2} -1}{h^\alpha_{k,2} } & 1
        &  \frac{h^\alpha_{k,3} -1}{h^\alpha_{k,3} } &  &   \cdots & \\
        0 &  & \cdots & & & 0 &  & \cdots &  & 0 &  \frac{1}{h^\alpha_{k,3} } &  &  \cdots & \\
        & & & \cdots & & & & & & \cdots & & & & & & \cdots & &   &  \\
        & & & \cdots & & & & & & \cdots & & & & & & \cdots & &   &  \\
        & & & & \cdots & & & & & \cdots &  & & & & &\cdots &  & \frac{1}{h^\alpha_{k, w_k} } & 0 \\
        & & & & \cdots & & & & & \cdots  & & & & & & \cdots &  & \frac{h^\alpha_{k, w_k} -1}{h^\alpha_{k, w_k} } & 1
       \end{array} \right] }. & \nonumber
   \end{eqnarray}
 In view of the above construction of $F_\alpha$ and the linear selection function $\hat b^\alpha$, we obtain the following proposition that characterizes the piecewise linear formulation of $\hat b$.

\begin{proposition}  \label{prop:PL_formulation}
 For each index set $\alpha \subseteq\{ 1, \ldots, K_n-1\}$, its corresponding linear selection function $\ol b^\alpha$ is given by
 \[
     \hat b^\alpha (\bar y)  \, = \, F^T_\alpha \big( \, F_\alpha \, \Lambda \, F^T_\alpha \, \big)^{-1} F_\alpha \bar y,
 \]
 where $F_\alpha \in \mathbb R^{\ell \times (K_n+1)}$  is a row independent matrix given in (\ref{eqn:F_alp}).
\end{proposition}

Let $F_\alpha  \Lambda F^T_\alpha  := G + \lambda H \in \mathbb R^{\ell \times \ell}$, where $ G := F_\alpha \Gamma F^T_\alpha$ and $H := (F_\alpha D^T_2) (F_\alpha D^T_2)^T $. The following result establishes the properties of the matrix $H$.

\begin{proposition} \label{prop:matrix_H}
  Let $m=2$. For any $K_n$ and $\alpha$, the matrix $H$ is a symmetric, banded matrix with the bandwidth $m$.  Furthermore, (i) $0\le H_{ss} \le 6$, $\forall \, s=1, \ldots, \ell$; (ii) $|H_{s(s+1)}|=|H_{(s+1)s}| \le 4$,  $\forall \, s=1, \ldots, \ell-1$; and (iii) $|H_{s(s+2)}| = |H_{(s+2)s}|\le 1$, $\forall \, s=1, \ldots, \ell-2$.
\end{proposition}

\begin{proof}
  Consider $F_\alpha$ in (\ref{eqn:F_alp}). Recall that if $m^\alpha_k=1$, then $F_{\alpha, k}=1$ and we define $w_k:=1$ and $h^\alpha_{k,1} :=1$; otherwise, $h^\alpha_{k, j} \ge 2$ for each $j=1, \ldots, w_k$.  It follows from the structure of  $F_\alpha$ in (\ref{eqn:F_alp}) and (\ref{eqn:F_k}) that
\begin{align}
\lefteqn{  F_\alpha D^T_2  =  }  \label{eqn:F_k*D^T_2} \\
    &  {\tiny \left[ \begin{array}{cccccccccccccccccccccccccc}
        \mathbf 0  &     \frac{1}{h^\alpha_{1,1}}   & \mathbf 0 & 0 & \cdots &  \cdots &  & \cdots &  &  \\
        \mathbf 0 & - \frac{1}{h^\alpha_{1, 1}}- \frac{1}{h^\alpha_{1, 2}}  & \mathbf 0  & \frac{1}{h^\alpha_{1, 2}}  &  \mathbf 0    & 0 &  \cdots      & \cdots  & \mathbf  0  &  0 &  & \cdots  \\
         \mathbf 0 &  \frac{1}{h^\alpha_{1, 2}}  & \mathbf 0 & - \frac{1}{h^\alpha_{1, 2}}- \frac{1}{h^\alpha_{1, 3}}  &  \mathbf 0  &  \frac{1}{h^\alpha_{1, 3}} & \cdots & \cdots  & \mathbf  0   &  0 &   \cdots   \\
         \mathbf 0 & 0  & \mathbf 0 &  \frac{1}{h^\alpha_{1, 3}}  &  \mathbf 0     &  - \frac{1}{h^\alpha_{1, 3}}- \frac{1}{h^\alpha_{1, 4}}     & \cdots  & \cdots  & \mathbf 0 &  0 &  &  \cdots & \\
    \mathbf 0 & 0  & \mathbf 0 &  0  &  \mathbf 0     &  \frac{1}{h^\alpha_{1, 4}}     & \cdots  & \cdots  & \mathbf 0 &  0 &  &  \cdots & \\
         &  \cdots  & & \cdots & \cdots & & & \cdots & & \cdots  & & \cdots    \\
        &   \cdots & & \cdots & \cdots & & &\cdots  & & \cdots  & &  \cdots    \\
          & \cdots & & & \cdots & & &  \cdots & \mathbf 0 &  \frac{1}{h^\alpha_{1, w_1} }    & 0  & & &\cdots    \\
        &   \cdots & & & \cdots & & &  \cdots & \mathbf 0 & - \frac{1}{h^\alpha_{1, w_1}} -1     &  1 & \mathbf 0 &  \cdots     \\
     &  \cdots & & & \cdots & & &  \cdots & {\mathbf 0} & 1   &  - \frac{1}{h^\alpha_{2, 1}}-1     & \mathbf 0  &  \frac{1}{h^\alpha_{2, 1} }   &\cdots     \\
  &  \cdots & & & \cdots & & &  \cdots &   & 0   &   \frac{1}{h^\alpha_{2, 1}}   & \mathbf 0  & -  \frac{1}{h^\alpha_{2, 1}} - \frac{1}{h^\alpha_{2, 2}}   & \cdots     \\
  &  \cdots & & \cdots & \cdots & & &\cdots  & & \cdots  & &  \cdots &  \frac{1}{h^\alpha_{2, 2}}  &   \\
    & \cdots & & \cdots & \cdots & & &\cdots  & & \cdots  & &  \cdots  &  &\cdots   \\
  &  \cdots & & \cdots & \cdots & & &\cdots  & & \cdots  & &  \cdots  &  &\cdots
       \end{array} \right] }, & \nonumber
   \end{align}
where the bold face $\bf 0$ in the above matrix denotes the zero row of $(h^\alpha_{k, j}-1)$ elements if the first nonzero term $1/h^\alpha_{k, j}$ (from the top) appears in the column immediately to its right.  (By convention, if $h^\alpha_{k, j}=1$, then the zero row vanishes.) For example, the zero row $\bf 0$ in the first column has $(h^\alpha_{1,1}-1)$ elements, and  the next $\bf 0$  in the third column has $(h^\alpha_{1,2}-1)$ elements.

Clearly, $H=(F_\alpha D^T_2) (F_\alpha D^T_2)^T$ is symmetric. In light of (\ref{eqn:F_k*D^T_2}), we see that $H_{ij}=0$ whenever $| i-j | \ge 3$. Furthermore, in view of $h^\alpha_{k, j} \ge 1$, we have for each suitable $s$,
\begin{eqnarray*}
  0 \le H_{ss} & \le  & \max_{(k, j)} \left ( 1 + \Big(1 + \frac{1}{h^\alpha_{k,j} } \Big)^2 + \Big( \frac{1}{h^\alpha_{k,j} } \Big)^2  \right) \, \le \, 6, \\
   | H_{s(s+1)}| & \le & 2 \max_{(k, j)}  \left(  \frac{1}{h^\alpha_{k,j}} +1 \right) \, \le \, 4, \\
    | H_{s(s+2)}| & \le &  \max_{(k, j)}  \left(  \frac{1}{h^\alpha_{k,j}  h^\alpha_{k,j+1} },  \,  \frac{1}{h^\alpha_{k,j}} \right) \, \le \, 1.
 \end{eqnarray*}
This yields the proposition.
\end{proof}

The next proposition further establishes important properties of the matrix $F_\alpha \Gamma F^T_\alpha$ that pave the way for  the uniform Lipschitz property of  $\hat b$.

\begin{proposition} \label{prop:FLamF_trans}
%
There exists $P \in \mathbb N$ such that  for any $K_n$  with $n/K_n \ge P$ and any index set $\alpha$, the matrix $G:= F_\alpha \, \Gamma \, F^T_\alpha$ is a symmetric, strictly diagonally dominant, and tridiagonal matrix.
\end{proposition}

\begin{proof}
  Clearly, $G$ is symmetric. For a given index subset $\beta^\alpha_k$ corresponding to $\alpha$ constructed above, let $\Gamma_{\beta^\alpha_k \beta^\alpha_k}$ denote the principal submatrix of $\Gamma$ defined by $\beta^\alpha_k$.  In view of the structure of $F_\alpha$ in (\ref{eqn:F_alp}), we have
\[
  G= \mbox{diag}\Big( F_{\alpha, 1} \Gamma_{\beta^\alpha_1 \beta^\alpha_1}F^T_{\alpha, 1}, \ F_{\alpha, 2} \Gamma_{\beta^\alpha_2 \beta^\alpha_2} F^T_{\alpha,2}, \ \ldots, \ F_{\alpha, L} \Gamma_{\beta^\alpha_L \beta^\alpha_L} F^T_{\alpha, L} \Big).
\]

In what follows, we drop the subscript $n$ in $\theta_n$ and $\eta_n$ for notational simplicity.
To determine the entries of $G=(g_{ij})$, we  consider a fixed $k\in \{1, \ldots, L\}$. If $m^\alpha_k=1$, then $F_{\alpha, k} \Gamma_{\beta^\alpha_k \beta^\alpha_k} F^T_{\alpha, k}$
  is a scalar that appears on the diagonal of $G$. Denoting this number by $g_{ss}$, we have
\[
    g_{ss}  \, = \,  F_{\alpha, k} \Gamma_{\beta^\alpha_k \beta^\alpha_k} F^T_{\alpha, k} \, = \,  \left \{
     \begin{array}{llcc} \theta, \quad & \mbox{ if } k\in \{1, L\} \\
                            1,   \quad & \mbox{ otherwise}
    \end{array} \right.
\]
and $g_{s(s+1)}=g_{(s+1)s}=\eta$, $g_{s j} =0$ for all $j$ with $|s-j| \ge 2$.

Now consider  $m^\alpha_k>1$. In this case, $F_{\alpha, k} \Gamma_{\beta^\alpha_k
\beta^\alpha_k}F^T_{\alpha, k} $ is a symmetric and positive definite
matrix of order $(w_k+1)$ that forms a diagonal block of $G$. Making use of the structure of $F_{\alpha, k}$ given in (\ref{eqn:F_k}) and somewhat lengthy computation, we obtain the following results in two separate cases:
\begin{itemize}
  \item [(1)] $k=1$ or $k=L$. For $k=1$,
  \begin{eqnarray*}
      g_{11} & = & \theta+ \eta - \frac{\eta}{h^\alpha_{1,1}}+ (1+2\eta)
          \frac{ (h^\alpha_{1,1}-1)(2h^\alpha_{1,1}-1)}{6 h^\alpha_{1,1} }, \label{eqn:d11} \\
       g_{s(s+1)} & = &   g_{(s+1)s} \, = \,
          \frac{ \eta }{ h^\alpha_{1, s} } + (1+2\eta) \frac{ (h^\alpha_{1, s})^2-1 }
           {6 h^\alpha_{1, s}}, \quad \forall \ s=1, \ldots, w_1, \\
      g_{ss} & = & (1+2\eta) \left [ \frac{2 (h^\alpha_{1,s-1})^2+1 }{6 h^\alpha_{1, s-1} }
        + \frac{ 2 (h^\alpha_{1,s})^2+1 }{6 h^\alpha_{1, s} } \right ]
        - \left( \frac{1}{h^\alpha_{1, s-1} }+  \frac{1}{h^\alpha_{1, s}} \right) \eta, 
         \  \forall \, s=2, \ldots, w_1,  \\
      g_{(w_1+1)(w_1+1)} & = & (1+2\eta) \frac{ (h^\alpha_{1, w_1} +1)(2h^\alpha_{1, w_1}+1) }
      {6 h^\alpha_{1, w_1} }- \left(1+\frac{1}{h^\alpha_{1, w_1}  } \right) \eta. \label{eqn:d22}
  \end{eqnarray*}
  Besides, $g_{(w_1+1)(w_1+2)}=g_{(w_1+2)(w_1+1)}=\eta$, and
  for each $s=1, \ldots, w_1$,
  $g_{s j}=0$ once $|s-j| \ge 2$.
  For $k=L$, the similar
   results can be established by using the symmetry of the rows of $F_{\alpha, L}$.

 \item [(2)] $k\in \{ 2, \ldots, L-1\}$. In this case, suppose
 that
 the $(1,1)$-element of $F_{\alpha, k} \Gamma_{\beta^\alpha_k \beta^\alpha_k} F^T_{\alpha,
 k}$ corresponds to the diagonal entry $g_{tt}$ of $G$, where $2\le t \le \ell-1$. Then we have
   \begin{eqnarray*}
     g_{tt} & = & 1+ \eta - \frac{\eta}{h^\alpha_{k,1}}+ (1+2\eta)
          \frac{ (h^\alpha_{k,1}-1)(2h^\alpha_{k,1}-1)}{6 h^\alpha_{k,1}  },  \label{eqn:dss} \\
     g_{(t+s-1)(t+s)} & = &  g_{(t+s)(t+s-1)}   \, = \,
            \frac{ \eta }{ h^\alpha_{k, s} } + (1+2\eta) \frac{ (h^\alpha_{k, s})^2-1 }
           {6 h^\alpha_{k, s}}, \ \ \forall \, s=1, \ldots, w_k, \\
    g_{(t+s)(t+s)} & = & (1+2\eta) \left [ \frac{2 (h^\alpha_{k,s+1})^2+1 }{6
      h^\alpha_{k, s+1} }
        + \frac{ 2 (h^\alpha_{k,s})^2+1 }{6 h^\alpha_{k, s} } \right ]
        - \left ( \frac{1}{h^\alpha_{k, s+1} }+  \frac{1}{h^\alpha_{k, s}} \right ) \eta, \nonumber \\
         & & \qquad \  \forall \ s=1, \ldots, w_k-1,  \\
    g_{(t+w_k)(t+w_k)} & = &  (1+2\eta) \frac{ (h^\alpha_{k, w_k} +1)(2h^\alpha_{k, w_k}+1) }
      {6 h^\alpha_{k, w_k} }  - \left (1+\frac{1}{h^\alpha_{k, w_k}  } \right ) \eta. \label{eqn:ds+1}
   \end{eqnarray*}
    In addition,   $g_{t(t-1)}=g_{(t+w_k)(t+w_k+1)}=\eta$, and
     for each $s=t, \ldots, t+w_k+1$,    $g_{s j}=0$ once $|s-j| \ge 2$.
 \end{itemize}
 These results show that $G$ is tridiagonal.

Next we show that $G \in \mathbb R^{\ell \times \ell}$ is strictly diagonally dominant.  Define
\[
  \xi_1 := g_{11}-|g_{12}|, \  \xi_\ell:= g_{\ell\ell}-| g_{\ell(\ell-1)}|,  \  \mbox{ and } \
   \xi_i := g_{ii} - |g_{i(i-1)}|-|  g_{i(i+1)}|, \  \forall \, i\in\{2,\ldots, \ell-1\}.
\]
Recall that $\theta \rightarrow 1/2$ and $\eta \rightarrow 1/4$ as $n/K_n \rightarrow \infty$.
Hence, there exists $P \in \mathbb N$ such that for any $K_n$ with $n/K_n \ge P$ and any $\alpha$ such that
%
%
%
%
  \begin{itemize}
      \item [(3.1)]  if $m^\alpha_k=1$, then (i)  if $k\in \{1, L\}$, the corresponding
       $\xi_i=\theta - |\eta| > 1/5$;
       and (ii) otherwise, the corresponding $\xi_i=1 - 2|\eta| > 1/3$.

      \item [(3.2)] if $m^\alpha_k>1$ with $k=1$, then\\
       (i) for $s=1$, in view of $ h^\alpha_{1,1} \ge 2$, the corresponding
       \begin{eqnarray*}
       \xi_i & = & g_{11} - |g_{12}|  =  \theta+ \eta - \frac{\eta}{h^\alpha_{1,1}}+ (1+2\eta)
          \frac{ (h^\alpha_{1,1}-1)(2h^\alpha_{1,1}-1)}{6 h^\alpha_{1,1} } - \left | \frac{ \eta }{ h^\alpha_{1, 1} } + (1+2\eta) \frac{ (h^\alpha_{1, 1})^2-1 } {6 h^\alpha_{1, 1}} \right | \\
    & = &  \theta+ \Big(1  - {  2  \over h^\alpha_{1,1} } \Big) \eta  +  (1+2\eta) \left [
          \frac{ (h^\alpha_{1,1}-1)(2h^\alpha_{1,1}-1)}{6 h^\alpha_{1,1} }  - \frac{ (h^\alpha_{1, 1})^2-1 } {6 h^\alpha_{1, 1}} \right ] \\
       & \ge &  \theta  + (1+2 \eta ) \Big( \frac{ h^\alpha_{1,1} } {6} - {1 \over 2} + {1 \over 3  h^\alpha_{1,1} } \Big) \,  \ge \,   \theta  + (1+2 \eta ) \Big( \frac{ h^\alpha_{1,1} } {6} - {1 \over 2}  \Big) \\
     &   \ge &   \theta + (1+2 \eta) \Big( \frac{ h^\alpha_{1,1} } {42}  + {2 \over 7} - {1 \over 2}  \Big)  \\
      &  \ge &
           {1 \over 14} + (1+2\eta)  \frac{h^\alpha_{1,1}}{42}  \, > \, 0.
         \end{eqnarray*}
       (ii) for $s=2, \ldots, w_1$, the corresponding
       \begin{eqnarray*}
        \xi_i & = & g_{ss}-|g_{s(s-1)}| -|g_{s(s+1)}| \\
        & = & (1+2\eta) \left [ \frac{2 (h^\alpha_{1,s-1})^2+1 }{6 h^\alpha_{1, s-1} }
        + \frac{ 2 (h^\alpha_{1,s})^2+1 }{6 h^\alpha_{1, s} }  - \frac{ (h^\alpha_{1, s})^2-1 }
           {6 h^\alpha_{1, s}}   -  \frac{ (h^\alpha_{1, s-1})^2-1 }
           {6 h^\alpha_{1, s-1}}  \right ]  \\
           & &  \quad - 2 \left( \frac{1}{h^\alpha_{1, s-1} }+  \frac{1}{h^\alpha_{1, s}} \right) \eta\\
      & = &
        (1+ 2 \eta)  \left (
          \frac{ h^\alpha_{1,s-1} + h^\alpha_{1,s} }{6} + \Big[{1 \over  h^\alpha_{1,s-1} } +  {1 \over  h^\alpha_{1,s} } \Big]\Big[ {1\over 3} - { 2 \eta \over 1 +2 \eta} \Big] \right)  \\
       & \ge &
       (1 + 2 \eta)   \left (   \frac{ h^\alpha_{1,s-1} + h^\alpha_{1,s} }{8} \right)       \,  > \, 0,
        \end{eqnarray*}
     where the last inequality follows from the facts that ${ 2 \eta \over 1 +2 \eta} \rightarrow 1/3$ as $n/K_n \rightarrow \infty$, and that for any $h_1, h_2 \ge 2$,
     \[
            \frac{1}{h_1} + \frac{1}{h_2} \, \le \, \frac{ h_1 + h_2} {4 }.
     \]
         and (iii) for $s=w_1+1$, the corresponding
         \begin{eqnarray*}
          \xi_i & = & g_{(w_1+1)(w_1+1)} - | g_{(w_1+1)w_1}| -|g_{(w_1+1)(w_1+2)}| \\
             & = &  (1+2\eta) \frac{ (h^\alpha_{1, w_1} +1)(2h^\alpha_{1, w_1}+1) }
      {6 h^\alpha_{1, w_1} }- \left(1+\frac{1}{h^\alpha_{1, w_1}  } \right) \eta - \left |
      \frac{ \eta }{ h^\alpha_{1, w_1} } + (1+2\eta) \frac{ (h^\alpha_{1, w_1})^2-1 }
           {6 h^\alpha_{1, w_1}}  \right |  - \left | \eta \right | \\
        &   = & (1+2\eta) \left [  \frac{ (h^\alpha_{1, w_1} +1)(2h^\alpha_{1, w_1}+1) }
      {6 h^\alpha_{1, w_1} } -  \frac{ (h^\alpha_{1, w_1})^2-1 } {6 h^\alpha_{1, w_1}} \right]  - 2 \left(1+\frac{1}{h^\alpha_{1, w_1}  } \right) \eta \\
      & = &  ( 1+ 2 \eta)  \left (  \frac{ h^\alpha_{1, w_1} }{6} + {1 \over  3h^\alpha_{1, w_1} } + {1 \over 2}  - \Big(1+ \frac{1}{ h^\alpha_{1, w_1} } \Big) { 2 \eta \over 1 +2 \eta}    \right )  \\
       & \ge &   ( 1+ 2 \eta)  \left (  \frac{ h^\alpha_{1, w_1} }{8}   +  {1 \over 8}  \right)
        \, > \, 0.
   %
       \end{eqnarray*}
        The similar results can be obtained for $m^\alpha_k>1$ with
        $k=L$ using symmetry.

      \item [(3.3)] if $m^\alpha_k>1$ with $k\in\{ 2, \ldots,
      L-1\}$, then  \\
      (i) for the $t$th row,  the corresponding
        \begin{eqnarray*}
          \xi_i & = & g_{tt} - |g_{t(t-1)}| -|g_{t(t+1)}| \\
              & = & 1+ \eta - \frac{\eta}{h^\alpha_{k,1}}+ (1+2\eta)
          \frac{ (h^\alpha_{k,1}-1)(2h^\alpha_{k,1}-1)}{6 h^\alpha_{k,1}  } - | \eta | - \left |   \frac{ \eta }{ h^\alpha_{k, 1} } + (1+2\eta) \frac{ (h^\alpha_{k, 1})^2-1 }
           {6 h^\alpha_{k, 1}} \right |  \\
            &  = & \Big(  1-  \frac{2 \eta}{h^\alpha_{k,1}} \Big) + (1+2\eta) \left [ \frac{ (h^\alpha_{k,1}-1)(2h^\alpha_{k,1}-1)}{6 h^\alpha_{k,1}  }- \frac{ (h^\alpha_{k, 1})^2-1 }
           {6 h^\alpha_{k, 1}}  \right] \\
            &  \ge &  \big( 1 - \eta \big) +  ( 1+ 2 \eta)  \left (  \frac{ h^\alpha_{k, 1} }{6} + {1 \over  3h^\alpha_{k, 1} } - {1 \over 2}  \right )  \, \ge \,  ( 1+ 2 \eta)  \left (    \frac{1 - \eta  }{ 1+ 2 \eta }  + \frac{ h^\alpha_{k, 1} }{6}  - {1 \over 2}    \right ) \\
            & \ge &
           (1+ 2 \eta)   \left (  \frac{1 -  \eta  }{ 1+ 2 \eta } + \frac{ h^\alpha_{k, 1} }{7} + {1 \over 21}    -  {1 \over 2}    \right )
          \,   \ge \, ( 1+ 2 \eta)  \left (  \frac{ h^\alpha_{k, 1} }{7} + {1 \over 42}    \right )
            >0
            %
         \end{eqnarray*}
      for all $n/K_n$ sufficiently large. \\
       (ii) for $s=1, \ldots, w_k-1$, the corresponding
        \begin{eqnarray*}
            \xi_i & = & g_{(t+s)(t+s)}-|g_{(t+s)(t+s-1)}| - |g_{(t+s)(t+s+1)}| \\
             & = &  (1+2\eta) \left ( \frac{2 (h^\alpha_{k,s+1})^2+1 }{6
      h^\alpha_{k, s+1} }
        + \frac{ 2 (h^\alpha_{k,s})^2+1 }{6 h^\alpha_{k, s} } \right )
        - \left ( \frac{1}{h^\alpha_{k, s+1} }+  \frac{1}{h^\alpha_{k, s}} \right ) \eta \\
            & &  \quad    \, - \, \left |   \frac{ \eta }{ h^\alpha_{k, s} } + (1+2\eta) \frac{ (h^\alpha_{k, s})^2-1 } {6 h^\alpha_{k, s}} \right |
         \, - \, \left |   \frac{ \eta }{ h^\alpha_{k, s+1} } + (1+2\eta) \frac{ (h^\alpha_{k, s+1})^2-1 } {6 h^\alpha_{k, s+1}} \right | \\
        &=  &   (1+2\eta) \left[  \frac{2 (h^\alpha_{k,s+1})^2+1 }{6
      h^\alpha_{k, s+1} }
        + \frac{ 2 (h^\alpha_{k,s})^2+1 }{6 h^\alpha_{k, s} } -    \frac{ (h^\alpha_{k, s})^2-1 }
           {6 h^\alpha_{k, s}}  - \frac{ (h^\alpha_{k, s+1})^2-1 }
           {6 h^\alpha_{k, s+1}}    \right]  \\
          & &  \quad   \, - \,  2 \left ( \frac{1}{h^\alpha_{k, s+1} }+  \frac{1}{h^\alpha_{k, s}} \right ) \eta \\
        & = & ( 1+ 2 \eta)   \left (
          \frac{ h^\alpha_{k,s+1} + h^\alpha_{k,s} }{6} + {1 \over  3h^\alpha_{k,s+1} } +  {1 \over  3h^\alpha_{k,s} }  - \Big[ \frac{1}{h^\alpha_{k, s+1} }+  \frac{1}{h^\alpha_{k, s}}  \Big]  { 2 \eta \over 1+ 2 \eta    }  \right)  \\
          & \ge & ( 1+ 2 \eta)   \left (  \frac{ h^\alpha_{k,s+1} + h^\alpha_{k,s} }{8}  \right)
           >0.
       \end{eqnarray*}
         and
         (iii) for $s=w_k$,  the corresponding
         \begin{eqnarray*}
         \xi_i &= & g_{(t+w_k)(t+w_k)} - | g_{(t+w_k)(t+w_k-1)}| -|g_{(t+w_k)(t+w_k+1)}| \\
         & = & (1+2\eta) \frac{ (h^\alpha_{k, w_k} +1)(2h^\alpha_{k, w_k}+1) }
      {6 h^\alpha_{k, w_k} }  - \left (1+\frac{1}{h^\alpha_{k, w_k}  } \right ) \eta  - \left |   \frac{ \eta }{ h^\alpha_{k, w_k} } + (1+2\eta) \frac{ (h^\alpha_{k,  w_k})^2-1 }
           {6 h^\alpha_{k, w_k}} \right | - | \eta |  \\
         & = & (1+2\eta)  \left [  \frac{ (h^\alpha_{k, w_k} +1)(2h^\alpha_{k, w_k}+1) }
      {6 h^\alpha_{k, w_k} }  -   \frac{ (h^\alpha_{k,  w_k})^2-1 }
           {6 h^\alpha_{k, w_k}}  \right ]  - 2 \left (1+\frac{1}{h^\alpha_{k, w_k}  } \right ) \eta \\
        & = &   (1+2\eta)  \left (  \frac{ h^\alpha_{k, w_k} }{6} + {1 \over  3h^\alpha_{k, w_k} } + {1 \over 2}  - \Big(1+ \frac{1}{ h^\alpha_{k, w_k} } \Big) { 2 \eta \over 1 +2 \eta }     \right ) \\
         & \ge &  (1+2\eta)  \left (  \frac{ h^\alpha_{k, w_k} }{8} + {1 \over 8}  \right) >0.
         \end{eqnarray*}
  \end{itemize}
   Consequently, all $\xi_i >0$ so that $G$ is strictly diagonally dominant.
\end{proof}





With the help of the above results,  we are ready to  prove the uniform Lipschitz property for the optimal spline coefficient vector $\hat b$.

\begin{theorem} \label{thm:uniform_Lipschitz}
 %
 %
 %
 There exist positive constants $\ol \lambda$, $P \in \mathbb N$, and  $\kappa_\infty$ (which is independent of $K_n$ and $\lambda$) such that
for  any $\lambda\in [0, \ol \lambda]$ and $K_n$ with $n/K_n \ge P$,
  \[
     \big \| \hat b(\bar y^1)-\hat b (\bar y^2) \big \|_\infty \, \le \, \kappa_\infty \, \big \| \bar y^1 - \bar y^2 \big \|_\infty,     \quad \forall \ \bar y^1, \bar y^2  \in \mathbb R^{K_n+1}.
     \]
\end{theorem}

\begin{proof}
By Proposition~\ref{prop:PL_formulation}, it suffices to show the uniform bound of $\|F^T_\alpha (F_\alpha \Lambda F^T_\alpha)^{-1} F_\alpha\|_\infty$.  Given any $K_n$  and any $\alpha$,  recall that $F_\alpha \Lambda F^T_\alpha= G + \lambda H$, where $G$ is tridiagonal and $H$ is banded with bandwidth $m=2$. Hence, $F_\alpha \Lambda F^T_\alpha$ is banded with bandwidth $m=2$.
Define
\[
  \wt \xi_i  :=  |(F_\alpha \Lambda F^T_\alpha)_{ii}| - \sum_{j\ne i} |(F_\alpha \Lambda F^T_\alpha)_{ij}|, \quad  \forall \ i = 1,\ldots, \ell.
\]
Let $P \in \mathbb N$ be given in Proposition~\ref{prop:FLamF_trans}.  It follows from Propositions~\ref{prop:matrix_H} and \ref{prop:FLamF_trans} that for any $K_n$ with $n/K_n \ge P$, $\wt \xi_i \ge \xi_i - 16 \lambda$ for each $i$, where $\xi_i$ is defined in the proof of Proposition~\ref{prop:FLamF_trans}. Moreover, Proposition~\ref{prop:FLamF_trans} shows that $\xi_i$'s are uniformly bounded below by a positive constant for any $K_n$ and $\alpha$. Therefore, there exists $\ol\lambda >0$ (independent of $K_n$ and $\alpha$) such that for any $\lambda \in [0, \ol\lambda]$, $\wt \xi_i \ge \xi_i/2 >0, \forall \ i=1, \ldots, \ell$ for any $K_n$ and $\alpha$.

 Let the diagonal matrix $\Xi:=\mbox{diag}(\wt\xi^{-1}_1, \ldots, \wt\xi^{-1}_\ell) \in \mathbb R^{\ell\times \ell}$, which is clearly invertible.
   We  have
   \[
   \| F^T_\alpha    (F_\alpha \Lambda F^T_\alpha)^{-1}  F_\alpha\|_\infty \, = \, \| F^T_\alpha \cdot
   (\Xi F_\alpha \Lambda F^T_\alpha)^{-1} \cdot (\Xi F_\alpha)\|_\infty  \, \le \, \| F^T_\alpha\|_\infty \cdot
     \| (\Xi  F_\alpha \Lambda F^T_\alpha)^{-1}\|_\infty \cdot \|\Xi F_\alpha \|_\infty,
  \]
    where it is easy to verify $\| F^T_\alpha\|_\infty=1$. Since $E:=\Xi F_\alpha \Lambda F^T_\alpha$ is strictly diagonally
   dominant with $E_{ii}-\sum^\ell_{j=1, j\ne i} |E_{ij}| =1$ for each $i$, it
   follows from the Ahlberg-Nilson-Varah bound \cite{Varah_LAA75}
    that $\| E^{-1} \|_\infty = \| (\Xi F_\alpha \Lambda F^T_\alpha)^{-1} \|_\infty  \le 1$.

 Next we establish the bound on $\| \Xi F_\alpha\|_\infty$ as follows. By using the results for $\xi_i$ developed in Proposition~\ref{prop:FLamF_trans} and $\wt\xi_i \ge \xi_i/2$, we have
for any $\alpha$, any $\lambda\in [0, \ol \lambda]$, and any $K_n$ with $n/K_n \ge P $,
   \begin{itemize}
      \item [(1)]    if $m^\alpha_k=1$, then the absolute sum of
      the entries in  the corresponding row in $\Xi F_\alpha $ is given by
        $1/\wt \xi_i \le 10$.
      \item [(2)] if $m^\alpha_k>1$ with $k=1$, then  \\
       (2.1) in view of  $h_{1,1} \ge 2$,  the absolute sum
      of the entries in the  row in  $\Xi F_\alpha $ corresponding to $g_{11}$  is given by
      \[
        \frac{1+h^\alpha_{1,1}}{2\wt \xi_i} \, \le \, \frac{1+h^\alpha_{1,1}} {
              {1 \over 14} + (1+2 \eta) \frac{h^\alpha_{1,1}}{42 } }         %
         \, \le \,  40.
        \]
       (2.2) for $s=2, \ldots, w_1$, the absolute sum of the entries in
        the  row in  $\Xi F_\alpha $ corresponding to $g_{ss}$ is
        given by
       \[
        \frac{h^\alpha_{1,s-1}+ h^\alpha_{1,s}}{2\wt \xi_i}
          \, \le \,  \frac{h^\alpha_{1,s-1}+ h^\alpha_{1,s}}{
           ( 1+ 2  \eta)   \left (   \frac{ h^\alpha_{1,s-1} + h^\alpha_{1,s} }{8}  \right)  }
        \, \le \,  12.
        \]
          (2.3) for $s=w_1+1$, the absolute sum
        of the entries in
        the  row in  $\Xi F_\alpha $ corresponding to $g_{(w_1+1)(w_1+1) }$
        is given by
        \[
            \frac{1+h^\alpha_{1,w_1}}{2\wt \xi_i}  \, \le \, \frac{1+h^\alpha_{1,w_1}}
           {  ( 1+ 2 \eta)  \left (  \frac{ h^\alpha_{1, w_1} }{8}  +  {1 \over 8}  \right)  }
            \, \le \,  12.
        \]
        The same results can be obtained for $m^\alpha_k>1$ with
        $k=L$.
   \item [(3)] if $m^\alpha_k>1$ with $k\in\{ 2, \ldots, L-1\}$, then  \\
      (3.1) the absolute sum of the entries in the  row in  $\Xi F_\alpha $ corresponding    to $g_{tt}$  is   given by
       \[
         \frac{1+h^\alpha_{k,1}}{2\wt \xi_i}  \, \le  \, \frac{1+h^\alpha_{k,1}}{ (1+ 2\eta) \left (  \frac{ h^\alpha_{k, 1} }{7 } + {1 \over 42}   \right ) }
           \, \le \, 40.
        \]
         (3.2) for $s=1, \ldots, w_k-1$, the absolute sum
        of the entries in
        the  row in  $\Xi F_\alpha $ corresponding to $g_{(t+s)(t+s)}$ is
        given by
        \[
         \frac{ h^\alpha_{k,s} + h^\alpha_{k,s+1} }{2\wt\xi_i} \, \le \,
          \frac{ h^\alpha_{k,s} + h^\alpha_{k,s+1} }{ (1+2\eta)  \left (
          \frac{ h^\alpha_{k,s+1} + h^\alpha_{k,s} }{8}  \right )  }
           \, \le \,  12.
        \]
        (3.3) for $s=w_k$, the absolute sum of the entries in
        the  row in  $\Xi F_\alpha $ corresponding to $g_{(t+w_k)(t+w_k) }$
        is given by
        \[
         \frac{1+h^\alpha_{k,w_k}}{2\wt\xi_i} \, \le \,  \frac{1+h^\alpha_{k,w_k}}{ (1+2\eta)   \left (  \frac{ h^\alpha_{k, w_k} }{8} + {1 \over 8}  \right) }
          \, \le \, 12.
        \]
  \end{itemize}

    In view of the above results, we deduce that the existence of a positive constant $\kappa_\infty$, which is independent of $K_n$, $\alpha$, and $\lambda$, such that
    for all $\lambda \in [0, \ol \lambda]$ and $n/K_n \ge P$,
     $\| \Xi F_\alpha\|_\infty \le  \kappa_\infty$. This in turn implies that    $\| F^T_\alpha   (F_\alpha \Lambda F^T_\alpha)^{-1} F_\alpha\|_\infty \le
    \kappa_\infty$, regardless of $\alpha$, $\lambda$, and $K_n$. Hence,  $\| \hat b(\bar y)\|_\infty \le \kappa_\infty \| \bar y \|_\infty$ for any $\bar y \in \mathbb R^{K_n+1}$.
  Finally,  the uniform Lipschitz property of $\hat b$ thus follows from the piecewise   linear property of $\hat b$ \cite[Proposition 4.2.2]{FPang_book03}.
\end{proof}

\begin{remark} \rm
 The uniform Lipschitz property established in Theorem~\ref{thm:uniform_Lipschitz} for $m=2$ can be extended to other difference penalties. In fact, it follows from the similar argument as in Proposition~\ref{prop:matrix_H} that for $m \ge 3$, the matrix $H$ is a banded matrix with the bandwidth $m$ and the absolute row sum of $H$ is uniformly bounded. Hence, by choosing a suitable $\ol \lambda>0$ sufficiently small, the uniform  Lipschitz property holds. More involved computations show that the uniform Lipschitz property also holds for the first order difference penalty. Nevertheless, it is the second order difference penalty that allows us to obtain the optimal rate of convergence uniformly on $\mathcal C_H(r, L)$ as shown in the next section.
\end{remark}

%
\subsection{Convex $P$-spline Estimator: Optimal Rate of Convergence}  \label{sect:rate_optimal}

In this section, we show that the proposed convex $P$-spline estimator $\hat f^{[1]}$ achieves the optimal convergence rate in the sup-norm, via the uniform Lipschitz property  (cf. Theorem~\ref{thm:uniform_Lipschitz}) and asymptotic estimation techniques.
To this end, we introduce the following convex $P$-spline function with $p=1$ based on noise free data $\vec f:=( f(x_1), \ldots, f(x_n))^T \in \mathbb R^{n}$, i.e.,
\[
   \bar f^{[1]} (x)  \, = \, \sum^{K_n+1}_{k=1} \check b_k B^{[1]}_k(x),
\]
where the linear spline coefficient vector $\check b =(\check b_1, \ldots, \check b_{K_n+1})^T \in \mathbb R^{K_n+1}$ is given by
\[
     \check b \, = \, \arg\min_{b\in \Omega} \, {1\over 2}\, b^T \Lambda \, b -   b^T \big( { X^T \vec f / \beta_n} \big).
\]
Here $\Omega$, $\Lambda$, and $\beta_n$ are defined before (\ref{eqn:bm2}). In other words, $\check b = \hat b( \mathbb E (\bar{y}) )$.
%
%

The following two propositions establish the uniform bounds for bias and stochastic errors of the proposed convex estimator, respectively.  The proof of Proposition~\ref{prop:stochastic} is similar to  that of  \cite[Proposition 4.2]{WangShen_SICONConvex12} recently established by the authors, and we present its proof to be self-contained and complete.
%
%
In what follows, let  the sup-norm $\| g \|_\infty :=\sup_{x\in [0, 1]} | g(x) |$ for a function $g \in C([0, 1])$.

\begin{proposition} \label{prop:bias}
Let  $r\in (1,  2]$. Then there exist two positive constants $C_1$ and $C_2$
%
%
such that for all sufficiently large $K_n$ with $n/K_n \ge P$ and for each $\lambda\in [0, \ol \lambda]$ (uniformly in $f$),
\begin{equation}\label{equ:bias}
   \sup_{f\in {\cal C}_H(r, L)}\|  \bar f^{[1]} - f\|_\infty \, \le \, C_{1} L  K_n^{-r} +  C_2 \sqrt{\lambda \cdot K_n}  \cdot L K^{-r}_n.
%
\end{equation}
\end{proposition}

\begin{proof}
  Given a function $f \in {\cal C}_H(r, L)$, we introduce the following functions, in addition to $\bar f^{[1]}$ defined above:
  \begin{itemize}
    \item [(i)] $\tilde f: [0, 1]\rightarrow \mathbb R$ is defined by the linear interpolation of $(\kappa_k, f(\kappa_k))$ with $k=0, 1, \ldots, K_n$. Hence, $\tilde f(\kappa_k)=f(\kappa_k)$ for all $k$.
    \item [(ii)]  Let $\vec{\tilde f}:=(\tilde f(x_1), \ldots, \tilde f(x_n))^T \in \mathbb R^n$, and define $\grave f:[0, 1] \rightarrow \mathbb R$ as  $\grave f (x) := \sum^{K_n+1}_{k=1} \grave b_k B^{[1]}_k(x)$, where $\grave b:=( \grave b_1, \ldots, \grave b_{K_n+1})$ is  given by $  \grave b \, = \, \arg\min_{b\in \Omega} \, {1\over 2}\, b^T \Lambda \, b -   b^T \big( { X^T \vec{\tilde f} / \beta_n} \big)$.
  \end{itemize}
 Clearly, $\|   \bar f^{[1]} - f \|_\infty \le \| \bar f^{[1]} - \grave f \|_\infty + \| \grave f - \tilde f \|_\infty + \| \tilde  f - f \|_\infty$, and we obtain the uniform bounds for each term on the right-hand side as follows:
 \begin{itemize}
   \item [(1)] $ \| \tilde  f - f \|_\infty$.  For each $x\in [\kappa_{k-1}, \kappa_k]$, $k=1, \ldots, K_n$, there exist $\xi_x, \tilde \xi_x \in (\kappa_{k-1}, \kappa_k)$ such that
\begin{eqnarray*}
\tilde f(x) - f(x) & = & \big[ \tilde f(\kappa_{k-1}) + K_n \big (\tilde f(\kappa_k) - \tilde f(\kappa_{k-1}) \big)(x-\kappa_{k-1}) \big] - \big[ f(\kappa_{k-1})+f'(\xi_x)(x-\kappa_{k-1})\big] \\
& = & \big( f'(\tilde\xi_x) - f'(\xi_x) \big) (x-\kappa_{k-1}).
%
\end{eqnarray*}
Thus by the H\"older condition, we have, for any $x\in [\kappa_{k-1}, \kappa_k]$, $| \tilde f(x) - f(x) | \le  L |\tilde\xi_x -\xi_x |^{\gamma} |x-\kappa_{k-1}|  \  \le \ L K_n^{-r}.$ This shows  $\|\tilde f - f\|_\infty \le L K_n^{-r}$.
   \item [(2)]  $\| \bar f^{[1]} - \grave f \|_\infty$.  It follows from the linear B-spline property that there exists a positive constant $\varrho$ (independent of $n$) such that $  \| X^T \|_\infty/\beta_n \le \varrho$ for all $n$ .   Since both $\bar f^{[1]}$ and  $\grave f$ are piecewise linear functions,  we deduce, via the uniform Lipschitz property of the optimal spline coefficient (cf. Theorem~\ref{thm:uniform_Lipschitz}), that for all $n/K_n$ and $\lambda>0$ sufficiently large,
   \[
       \| \bar f^{[1]} - \grave f \|_\infty \, \le \, \| \check b - \grave b \|_\infty \, \le \, \kappa_\infty \left    \| X^T ( \vec f - \vec{\tilde f} )/\beta_n \right \|_\infty \, \le \, \kappa_\infty \frac{  \| X^T \|_\infty}{\beta_n } \| f - \tilde f \|_\infty \, \le \, \kappa_\infty \varrho L K^{-r}_n.
   \]
   \item [(3)]  $\| \grave f - \tilde f \|_\infty$. Let the vector $\tilde b : =(\tilde f(\kappa_0), \tilde f(\kappa_1), \ldots, \tilde f(\kappa_{K_n}) )^T \in \mathbb R^{K_n+1}$. Since $\tilde f(\kappa_k) = f(\kappa_k)$ for all $k$ and $f$ is convex,  we have $D_2 \tilde b \ge 0$. Furthermore, due to the piecewise linear property of $\tilde f$, we deduce that $\tilde f(x) = \sum^{K_n+1}_{k=1} \tilde b_k B^{[1]}_k(x)$ and $\vec{\tilde f} = X \tilde b$. Along with the definition of the optimal spline coefficient vector $\check b$, this yields
   \[
       \frac{1}{2} \Big \| X \check b - \vec{\tilde f} \Big \|^2_2 + \lambda^* \big \| D_2 \check b \big \|^2_2 \, \le \,  \frac{1}{2} \left \| X \tilde b - \vec{\tilde f} \right \|^2_2 + \lambda^* \big \| D_2 \tilde b \big \|^2_2 \, = \, \lambda^* \big \| D_2 \tilde b \big \|^2_2.
   \]
  By the virtue of  $\vec{\tilde f} = X \tilde b$ and $\Gamma= X^T X/\beta_n \in \mathbb R^{(K_n+1)\times (K_n+1)}$, we further have
  \[
       \frac{1}{2} \big \| X (\check b -  \tilde b) \big \|^2_2 \, \le \, \lambda^* \big \| D_2 \tilde b \big \|^2_2 \ \Longleftrightarrow \    (\check b - \tilde b)^T \Gamma  (\check b - \tilde b) \le 2 \lambda \big \| D_2 \tilde b \big\|^2_2.
  \]
  %
   Since $\Gamma$ is symmetric and positive definite, it follows from  \cite[Lemma 6.2]{zhou_98} that  there exists a positive constant
$\mu$ (independent of $K_n$) such that  the smallest real eigenvalue of $\Gamma$ is bounded below by $\mu$ for any $K_n$. Therefore,
\[
 \mu \| \check b -\tilde b \|^2_\infty  \, \le \,  \mu  \| \check b -\tilde b \|^2_2  \, \le \,   (\check b - \tilde b)^T \Gamma  (\check b - \tilde b) \, \le \, 2 \lambda \big \| D_2 \tilde b \big\|^2_2.
\]
Moreover, using $\tilde b_k = f(\kappa_{k-1})$ and the H\"older condition for the true $f$, it is easy to show that
\[
  |\Delta^2 (\tilde b_k)| \, \le \, \frac{1}{K_n} L \big( 2 K_n)^{-\gamma} \, \le \, 2^{-\gamma} L K^{-r}_n, \qquad \forall \ k=2, \ldots, K_n+1.
\]
Hence, $\| D_2 \tilde b \|_2 \le (K_n)^{1/2}   L K^{-r}_n$. Letting $C_2:=(2/\mu)^{1/2}$, we have
\[
     \| \check b -\tilde b \|_\infty \, \le \,    \sqrt{ \frac{ 2 \lambda} { \mu } }  \, \| D_2 \tilde b\|_2 \, \le \,     \sqrt{ \frac{ 2 \lambda K_n } { \mu } } \cdot
       L K^{-r}_n  \, \le \, C_2 \sqrt{\lambda  K_n} \cdot L K^{-r}_n.
\]
  Consequently, $\| \grave f - \tilde f \|_\infty \le  \| \check b -\tilde b \|_\infty \le C_2 \sqrt{\lambda K_n}  \cdot L K^{-r}_n$.
 \end{itemize}
 Finally, putting the above uniform bounds together, we obtain the desired uniform bound of $ \| \bar f^{[1]} - f \|_\infty$ for all $f\in {\cal C}_H(r, L)$.
\end{proof}

%
%

\begin{proposition} \label{prop:stochastic}
  Let $r\in (1, 2]$, $\lambda \in (0, \ol \lambda]$, and $K_n$ satisfy $n/K_n \rightarrow \infty$ and $K_n /( n^{ {1\over 2r+1}} \sqrt{\log n}) \rightarrow 0$ as $n \rightarrow \infty$. Then there exists a positive constant $C_3$ such that for all $n$ sufficiently large (uniformly in $f$),
\[
  \sup_{ f \in \mathcal C_H(r, L)} \mathbb E\Big(\|\hat f^{[1]} - \bar f^{[1]} \|_\infty\Big) \, \le \,  C_3 \left ( { K_n \log n \over n } \right)^{1/2}.
\]
\end{proposition}

\begin{proof}
Define $\xi_k := \sum_{i=1}^n B_k^{[1]}(x_i)\epsilon_i/\sqrt{\beta_n}$, where $k=1, \ldots, K_n+1$. Hence $\xi_k \sim N(0, 1)$ for each $k=2, \ldots, K_n$.  Besides,   in view of $\sum^n_{i=1}(B^{[p]}_k(x_i))^2 \le \beta_n$,
we see that each $\xi_k$ with $k\in\{ 1, K_n+1\}$ has the normal distribution with mean zero and variance not greater than one. Hence, for any $t \ge 0$, $P( |\xi_{2}|\ge t ) \ge  P( |\xi_{k}|\ge t )$ for each $k\in\{ 1,  K_n+1\}$.
  Moreover, it follows from the uniform Lipschitz property of $\hat b$  (cf. Theorem~\ref{thm:uniform_Lipschitz}) and  (\ref{eqn:beta_n}) that
\[
  \|\hat f^{[1]} - \bar f^{[1]} \|_\infty \, \le \, {\sigma \kappa_{\infty}\over \sqrt{\beta_n}}\sup_{k=1, \ldots, K_n+1} |\xi_k| \, \le \, {\sigma \kappa_{\infty}\over \sqrt{C_{\beta}}}\sqrt{K_n\over n}\sup_{k=1, \ldots, K_n+1} |\xi_k|.
\]
Defining $C_4: = \kappa_{\infty}/\sqrt{C_{\beta}}$ and $ \ol \xi := \max_{k=1, \ldots, K_{n+1}}|\xi_k|$, we obtain that
  \[
\displaystyle \|\hat f^{[1]} - \bar f^{[1]}\|_\infty \, \le \,  \sigma  C_4 \sqrt{K_n\over n}~ \ol \xi,
\]
and that for any $u \ge 0$,  $P( |\xi_{2}|\ge { u\over C_4\sigma}\sqrt{n\over K_n} ) \ge  P( |\xi_{k}|\ge { u\over C_4 \sigma}\sqrt{n\over K_n} )$ for each $k\in\{ 1,  K_n+1\}$.
In light of all these results and the implication:
 $ Z\sim N(0, 1) \Longrightarrow P(Z\ge t)\le {1\over 2}e^{-t^2/2}, \forall \ t \ge 0$,  we deduce that for a given $u \ge 0$,
\begin{eqnarray*}
 P\Big( \|\hat f^{[1]} - \bar f^{[1]} \|_\infty \ge u \Big)    & \le & P\left ( \ol \xi \ge {u\over  C_4 \sigma}\sqrt{n\over K_n}   \, \right )
 \, \le \, (K_n+1) P\left (|\xi_2 |\ge { u\over C_4 \sigma}\sqrt{n\over K_n}  \, \right ) \\
& \le & (K_n+1) \exp\Big\{-{n\over 2 K_n  C_4^2\sigma^2}u^2 \Big\}.
\end{eqnarray*}

Let $T_n \, := \, C_4 \sigma \sqrt{2\over 2r+1} \sqrt{K_n \log n\over n} $. It follows from  the above result and
 $\int^\infty_{T} e^{-t^2/(2\sigma^2)} dt \le   \sigma e^{ -T^2/(2\sigma^2) } \sqrt{\pi/2}, \forall \, T\ge 0$  that for any $f \in \mathcal C_H(r, L)$,
\begin{eqnarray*}
\mathbb E\Big(\|\hat f^{[1]} - \bar f^{[1]} \|_\infty\Big)&\le & T_n + \int_{T_n}^\infty P\Big( \|\hat f^{[1]}  - \bar f^{[1]} \|_\infty>t   \Big)dt\\
&\le & T_n + \int_{T_n}^\infty (K_n+1)\exp\Big\{ -{n\over 2K_n  C_4^2\sigma^2} t^2 \Big\}dt\\
& \le & T_n + \sqrt{\pi\over 2}~C_4\sigma \sqrt{n^{-1}K_n}~ (K_n+1) n^{-{1\over 2r+1}} .
\end{eqnarray*}

Hence, if $K_n$ satisfies $\lim_{n\rightarrow \infty} K_n /( n^{ {1\over 2r+1}} \sqrt{\log n}) =0$, then $\mathbb E\big(\|\hat f^{[1]} - \bar f^{[1]}\|_\infty\big) = O(T_n)$ for all large $n$. This implies that there exists  a positive constant $C_3$ independent of $f$ such that for all $n$ sufficiently large,  $\mathbb E\big(\|\hat f^{[1]} - \bar f^{[1]}\|_\infty\big)   \le C_3   \sqrt{K_n \log n\over n} $ for any $f \in \mathcal C_H(r, L)$.
\end{proof}

%
%
%

The next theorem shows that by choosing suitable $K_n$ and $\lambda^*$, the proposed convex $P$-spline estimator achieves the optimal  rate of convergence in the sup-norm uniformly on  the function class ${\cal C}_H(r, L)$. This thus yields the desired minimax upper bound.

\begin{theorem} \label{thm:optimal_convergence_Pspline}
 If $K_n$ and $\lambda^*$ are chosen as
\[
   K_n \, = \, \left\lceil \Big({n\over \log n}\Big)^{1\over 2r+1} \right \rceil, \qquad
   \lambda^* \, = \, \beta_n \, K^{-1}_n,
\]
then there exists a positive constant $C_0$ (dependent on $\sigma, L, r$ only)  such that
\[
\sup_{f\in {\cal C}_H(r, L)} \mathbb E\Big( \|\hat f^{[1]} - f\|_\infty  \Big)   \, \le \, C_0  \Big({\log n\over  n}\Big)^{r\over 2r+1}, \qquad \forall \ \mbox{large} \ n.
\]
\end{theorem}

\begin{proof}
Clearly, the selected $K_n$ satisfies $K_n \rightarrow \infty$, $n/K_n \rightarrow \infty$ and $K_n /( n^{ {1\over 2r+1}} \sqrt{\log n}) \rightarrow 0$ as $n \rightarrow \infty$.
 It also follows from the choice of $\lambda^*$ that $\lambda = \lambda^*/ \beta_n = K^{-1}_n$.  Note that $\lambda \in (0, \ol \lambda]$ as $K_n \rightarrow \infty$.
 By Proposition~\ref{prop:bias}, we see that  the uniform bound of the bias is given by
 \[
  \sup_{f\in {\cal C}_H(r, L)}\|  \bar f^{[1]} - f\|_\infty \, \le \, C'_{1} L K_n^{-r},
\]
where $C'_1$ is a positive constant. Furthermore, in light of Proposition~\ref{prop:stochastic}, we have for all $n$ sufficiently large,
\[
\mathbb E \Big( \|\hat f^{[1]}  - f\|_\infty \Big) \, \le \, \|\bar f^{[1]} -   f \|_\infty + \mathbb E \Big(\|\hat f^{[1]} - \bar f^{[1]} \|_\infty \Big) \, \le \, C'_1 L K_n^{-r} + C_3 \sqrt{ K_n  \log n\over n}
\]
for any $f \in {\cal C}_H(r, L)$, where the positive constants $C'_1$ and $C_3$ are independent of $f$.  It is easy to show that for any fixed $n$, $\min_{K_n} \big(C'_1 L K_n^{-r} + C_3 ({K_n  \log n\over n})^{1/2}\big)$ is achieved when $C'_1 L K_n^{-r} =C_3 ({ K_n  \log n\over n})^{1/2}$. Simple calculation further shows that the choice of $K_n$ and $\lambda^*$ gives rise to the desired optimal convergence rate in the sup-norm.
%
%
\end{proof}

\begin{remark} \rm
Since the  convex  $P$-spline estimator achieves the uniform convergence on the entire interval $[0, 1]$, it is consistent not only in the interior of $[0, 1]$ but also on  the  boundary, which is a critical property that  many other convex estimators (e.g., the least squares  estimator) do not have. Roughly speaking, this is because the convex $P$-spline estimator  takes advantage of binned data between neighboring knots near a boundary point  to yield better estimates under suitable $\lambda_*$ and $K_n$, while other estimators do not do so. 
%
%
\end{remark}

%
%

%
\section{Concluding Remarks} \label{sect:conclusion}

In this paper, we have established the optimal rate of convergence for the minimax risk of convex estimators under the sup-norm.  The results developed in this paper shed light on further research on shape constrained minimax theory.  For example, the minimax lower bound is obtained via construction of a family of piecewise quadratic convex functions, and this approach can be  extended to other derivative constraints.  The minimax upper bound is developed by a convex $P$-spline estimator subject to the second order difference penalty.  A key step in the upper bound analysis is the uniform Lipschitz property for  optimal spline coefficients. This important property is known to hold for 
monotone $P$-splines \cite{SWang_SICON11}, and it is conjectured that the similar  property also holds for other shape constrained $P$-splines, but its proof is much more involved and shall be reported in the future.  
Other related  topics include confidence band construction for  shape constrained estimators \cite{dumbgen_03}.


\begin{thebibliography}{99}

{\small


\bibitem{Birge_TIT05}
 L. Birg\'{e}. A new lower bound for multiple hypothesis testing. {\it IEEE Trans. on Information Theory}, Vol. 51, pp. 1611--1615, 2005.

%
%


\bibitem{CaiLow_12} 
T. Cai and M. Low. A framework for estimation of convex functions.
Technical report, 2012.

%
%
%
\bibitem{CoverT_book05} 
T. M. Cover and J. A. Thomas.  {\it Elements of Information Theory}. Wiley, 2005.
%

\bibitem{CPStone_book92}
{ R. W.\ Cottle, J.-S.\ Pang, and R. E.\ Stone}. {\it The Linear
Complementarity Problem}. Academic Press Inc., (Cambridge 1992).


%
%

\bibitem{dumbgen_03} 
{L. D\"umbgen}. Optimal confidence bands for shape-restricted curves. {\it Bernoulli}, Vol. {9}, pp. 423--449, 2003.

\bibitem{dumbgen_04} 
{L. D\"umbgen, S. Freitag, and G. Jongbloed}. Consistency of concave regression with an application to current-status data. {\it Mathematical Methods of Statistics}, Vol. {13}, pp. 69--81, 2004.

%

%


%
%


\bibitem{EgMartin_book10}
{M. Egerstedt and C. Martin}. {\it Control Theoretic Splines}.
Princeton University Press, 2010.

%
%

\bibitem{FPang_book03}
{F. Facchinei and J.-S. Pang}. {\it Finite-Dimensional Variational Inequalities and Complementarity Problems}. Springer-Verlag, 2003.


\bibitem{Gallager_book68}
R. G. Gallager. {\it Information Theory and Reliable Communication}. Wiley, New York, 1968.

%
%
%
%
%

\bibitem{groeneboom_01}
{ P. Groeneboom, F. Jongbloed, and J. A. Wellner}.
Estimation of a convex function: Characterizations and asymptotic theory. {\it Annals of Statistics},  Vol. {29}, pp. 1653--1698, 2001.

%
%

\bibitem{hanson_76}
{P. L. Hanson and G. Pledger}. Consistency in concave regression. {\it Annals of Statistics}, Vol. {4},  pp.1038--1050, 1976.

%
%


\bibitem{hildreth_54}
C. Hildreth. Point estimates of ordinates of concave functions. {\it Journal of American Statistical Association}, Vol. {49}, pp. 598--619,  1954.

%
%


\bibitem{kiefer_82} 
{J. Kiefer}. Optimum rates for non-parametric density and regression estimates under order restrictions. In: Kallianpur, G., Krishnaiah, P. R., Ghosh, J.K. (Eds.), Statistics and Probability: Essays in honor of C. R. Rao. North-Holland, Amsterdam, pp. 419--428,  1982.

\bibitem{Kullback_TIT67}
S. Kullback. A lower bound for discrimination information in terms of variation. {\it IEEE Trans. on Information Theory}, Vol. 13, pp. 126--127, 1967.

%
\bibitem{LiRuppert_08}
{Y. Li and D. Ruppert}. On the asymptotics of penalized splines. {\it Biometrika}, Vol. {95}, pp. 415--436, 2008.
%
%
%

\bibitem{low_02} 
{M. Low and Y. Kang}. Estimating monotone functions. {\it Statistics \& Probability Letters}, Vol. {56}, pp. 361--367,  2002.

%
\bibitem{mammen_91}
{E. Mammen}.   Nonparametric regression under qualitative smoothness assumptions. {\it Annnals of  Statistics}, Vol. {19},
pp. 741--759, 1991.
%
%

\bibitem{mammen_99} 
{E. Mammen and C. Thomas-Agnan}. Smoothing splines and shape restrictions. {\it Scandinavian Journal of Statistics}, Vol. {26}, pp. 239-252, 1999.
%


\bibitem{marx_96}
{B. Marx and P. Eilers}. Flexible smoothing with B-splines and penalties (with comments and rejoinder). {\it Statistical Science}, Vol. {11}, pp. 89--121, 1996.

%
%

\bibitem{meyer_08}
{M. Meyer}. Inference using shape-restricted regression splines. {\it Annals of Applied Statistics}, Vol. {2},  pp. 1013--1033, 2008.

%
%

\bibitem{Nemirovski_notes00}
{A. Nemirovski}.  {\it Topics in Non-parametric Statistics}.  Lecture on Probability Theory and Statistics, Berlin, Germany: Springer-Verlag, Vol. 1738, Lecture Notes in Mathematics, 2000.



\bibitem{PW_Sinica07} 
{J. Pal and M. Woodroofe}. Large sample
properties of shape restricted regression estimators with
smoothness adjustments. {\it Statistica Sinica},  Vol. 17, pp. 1601--1616, 2007.

\bibitem{Pinsker_book64}
M. S. Pinsker. {\it Information and Information Stability of Random Variables and Processes}. Holden-Day, San Francisco, 1964.


\bibitem{ramsay_88a} 
{J. O. Ramsay}. Estimating smooth monotone
functions. {\it Journal of the Royal Statistical Society, Series
B}, Vol. {60}, pp. 365--375,  1988.

%
%
%

%
%
%
%
%


\bibitem{SWang_ACC10}
{J. Shen and X. Wang}. Estimation of shape constrained
functions in dynamical systems and its application to genetic
networks. {\it Proc. of American Control Conference}, pp.
5948--5953, Baltimore,  2010.

\bibitem{SWang_SICON11}
{J. Shen and X. Wang}. Estimation of monotone functions via
$P$-splines: A constrained dynamical optimization approach. {\it
SIAM Journal on Control and Optimization}, Vol. 49(2), pp.
646--671, 2011.


\bibitem{SWang_CDC11}
{J. Shen and X. Wang}.  A constrained optimal control approach
to smoothing splines.  {\it Proc. of the 50th IEEE Conf.
Decision and Control}, pp. 1729--1734,  Orlando, FL, 2011.

\bibitem{SWang_ACC12}
{J. Shen and X. Wang}. Convex regression via penalized splines: a complementarity approach. {\it Proc. of 2012 American Control Conference}, pp. 332--337, Montreal, Canada, 2012.


\bibitem{stone_82}
{C. J. Stone}. Optimal rate of convergence for nonparametric regression. {\it Annals of Statistics},  Vol. 10, pp. 1040--1053, 1982.

%
%

\bibitem{Tsybakov_book10} 
{A. B. Tsybakov}. {\it Introduction to Nonparametric Estimation}. Springer, 2010.

\bibitem{utreras_85} 
F. Utreras. Smoothing noisy data under monotonicity constraints: existence, characterization and convergence rates. {\it Numerische Mathematik}, Vol. {47}, pp. 611--625,  1985.


\bibitem{Varah_LAA75}
{J. M. Varah}. A lower bound for the smallest singular value of
a matrix. {\it Linear Algebra and Its Applications}. Vol. 11, pp.
3--5, 1975.

%
%
%

\bibitem{WangShen_Biometrika09}
{X. Wang and J. Shen}. A class of grouped Brunk estimators and
penalized spline estimators for monotone regression. {\it
Biometrika}, Vol. 97(3), pp. 585--601, 2010.

\bibitem{WangShen_SICONConvex12}
{X. Wang and J. Shen}. Uniform convergence and rate adaptive estimation of convex functions via constrained optimization.  In press, 2012.


%
%

\bibitem{woodroofe_93}
M. B. Woodroofe  and J. Sun. A penalized maximum likelihood estimate of $f(0_+)$ when $f$ is nonincreasing. {\it Statistica Sinica}, Vol. 3, pp. 501--515, 1993.


\bibitem{wright_81} 
{F. T. Wright}. The asymptotic behavior of
monotone regression estimates. {\it Annals of Statistics}, Vol. 9, pp. 443--448, 1981.

%
%


\bibitem{zhou_98}
{S. Zhou, X. Shen, and D. A. Wolfe}. Local asymptotics for
regression splines and confidence regions. {\it Annals of
Statistics},  Vol. 26, pp. 1760--1782,  1998.
%


}

\end{thebibliography}
\end{document}